\documentclass[11pt]{amsart}
\usepackage[square,sort&compress,comma,numbers]{natbib}
\usepackage{nicefrac,xcolor,upref,version}
\usepackage{amscd,amsthm,amsmath,amssymb,amsfonts}

\usepackage{mathrsfs}

\usepackage[colorlinks=true,citecolor=cyan,backref=page]{hyperref}

\usepackage{shuffle}
\usepackage{ulem}

\usepackage{enumerate}

\newtheorem{theorem}{Theorem}[section]
\newtheorem{lemma}[theorem]{Lemma} 
\newtheorem{proposition}[theorem]{Proposition} 
\newtheorem{definition}[theorem]{Definition} 
\newtheorem{problem}[theorem]{Problem} 
\newtheorem{corollary}[theorem]{Corollary}

\numberwithin{equation}{section}

\def\Q{{\mathbb {Q}}}
\def\Z{{\mathbb Z}}  
 
\def\C{{\mathbb C}}

\def\x{{\bf  x}}

\def\bc{{\bf c}}

\def\house#1{\setbox1=\hbox{$\,#1\,$}%
\dimen1=\ht1 \advance\dimen1 by 2pt \dimen2=\dp1 \advance\dimen2 by 2pt
\setbox1=\hbox{\vrule height\dimen1 depth\dimen2\box1\vrule}%
\setbox1=\vbox{\hrule\box1}%
\advance\dimen1 by .4pt \ht1=\dimen1
\advance\dimen2 by .4pt \dp1=\dimen2 \box1\relax}

\def\build#1_#2^#3{\mathrel{\mathop{\kern 0pt#1}\limitr_{#2}^{#3}}}

\def\date {le\ {\the\day}\ \ifcase\month\or 
janvier\or fevrier\or mars\or avril\or mai\or juin\or juillet\or
ao\^ut\or septembre\or octobre\or novembre\or 
d\'ecembre\fi\ {\oldstyle\the\year}}

\font\fivegoth=eufm5 \font\sevengoth=eufm7 \font\tengoth=eufm10

\newfam\gothfam \scriptscriptfont\gothfam=\fivegoth
\textfont\gothfam=\tengoth \scriptfont\gothfam=\sevengoth

\def\smallsquare{\vbox{\hrule\hbox{\vrule height 1 ex\kern 1 ex\vrule}\hrule}}

\def\bx{{\bf x}}
\def\s{{\bf s}}

\def\bfs{{\bf s}}
\def\bfv{{\bf v}}

\def\tit{{\tilde t}}
\def\tir{{\tilde r}}
\def\tis{{\tilde s}}

\def\rmi{{\rm i}}
\def\rme{{\rm e}} 

\def\us{{\underline s}} 
\def\usigma{{\underline \sigma}} 
\def\oplus{{\dot +}}
\def\ominus{{\dot -}}

\def \eqalign#1{\null\,\vcenter{\openup\jot\mathsurround 0pt\ialign{\strut
\hfil$\displaystyle{##}$&$\displaystyle{{}##}$\hfil
&&\quad\strut\hfil$\displaystyle{##}$&$\displaystyle{{}##}$\hfil
\crcr#1\crcr}}\,}

\def \ialign {\everycr {} \tabskip 0pt \halign }

\begin{document}

\title[Transcendence of values of Hecke--Mahler series]{Transcendence and continued fraction expansion of values of Hecke--Mahler series}

\author{Yann Bugeaud}
\address{Universit\'e de Strasbourg, Math\'ematiques,
7, rue Ren\'e Descartes, 67084 Strasbourg  (France)}
\address{Institut universitaire de France}
\email{bugeaud@math.unistra.fr}

\author{Michel Laurent} 
\address{Aix-Marseille Universit\'e, CNRS,  
Institut de Math\'ematiques de Marseille,
163 avenue de Luminy, Case 907, 
13288  Marseille C\'edex 9 (France)}
\email{michel-julien.laurent@univ-amu.fr}

\begin{abstract}
Let $\theta$ and $\rho$ be real numbers with
$0 \le \theta, \rho < 1$ and $\theta$ irrational. 
We show that the Hecke--Mahler series 
$$
F_{\theta, \rho} (z_1, z_2) 
= \sum_{k_1 \ge 1} \, \sum_{k_2 = 1}^{\lfloor k_1 \theta + \rho \rfloor} \, z_1^{k_1} z_2^{k_2},
$$ 
where $\lfloor \cdot \rfloor$ denotes the integer part function, 
takes transcendental values at any algebraic point $(\beta, \alpha)$ with 
$0 < |\beta|, |\beta \alpha^\theta |  < 1$.  
This extends earlier results of Mahler (1929) 
and Loxton and van der Poorten (1977), who settled the case $\rho=0$. 
Furthermore, for positive integers $b$ and  $a$, with $b \ge 2$ 
and $a$ congruent to $1$ modulo $b-1$,  
we give  the continued fraction expansion  of  the number  
 $$
 {(b-1)^2\over b} F_{\theta, \rho} \left({1\over b}, {1\over a}\right)+{\lfloor \theta+\rho\rfloor(b-1)\over b^2a},
 $$
  from which we derive 
a formula giving the irrationality exponent of $F_{\theta, \rho} (1/b, 1/a)$. 
\end{abstract}


\maketitle

\hfill{\it \`A la m\'emoire du Professeur Andrzej Schinzel}

\tableofcontents

\section{Introduction and main results} \label{intro}

Throughout, $\lfloor \cdot \rfloor$ and $\lceil \cdot \rceil$ are, respectively, the integer part and the upper integer part functions.   
For a real number $\theta$ in $(0, 1)$, set
$$
h_\theta (z) = \sum_{k \ge 1} \, \lfloor k \theta \rfloor z^k,
$$
where $z$ is a complex number with $|z| < 1$, and
$$
F_\theta (z_1, z_2) = \sum_{k_1 \ge 1} \, \sum_{k_2 = 1}^{\lfloor k_1 \theta \rfloor} \, z_1^{k_1} z_2^{k_2},
$$
where $z_1, z_2$ are complex numbers with 
$|z_1| < 1, |z_1 z_2^\theta | < 1$. 
The series $h_\theta (z)$ have been introduced by Hecke \cite{He22} in 1922.
B\"ohmer \cite{Bohm27} proved in 1927 that, if   
$\theta$ has unbounded partial quotients, then 
$h_\theta ( 1/b)$ is transcendental, for every integer $b \ge 2$.    
%, in his study of Dirichlet series. They 
%and were subsequently studied from a transcendence point of view by Mahler in 1929, who 
Two years later, in his fundational paper \cite{Mah29},   
Mahler introduced the two-variables series $F_\theta (z_1, z_2)$ (note that Mahler and most of his 
followers used $\omega$ in place of $\theta$, while we keep the notation from \cite{BuLa21})
and, among other results, he established that $h_\theta (\beta)$ is transcendental 
for every quadratic irrational number $\theta$ and every complex non-zero algebraic number 
$\beta$ in the open unit disc. 
This has been extended to every irrational number $\theta$ in $(0, 1)$ by Loxton and van der Poorten \cite{LovdP77c}
(see also \cite[Section 2.9]{Nish96}) nearly fifty years later.

We adopt a slightly different point of view to generalize the functions $h_\theta$ and $F_\theta$. 
Let $\theta$ and $\rho$ be real numbers with
$0 \le \theta, \rho < 1$ and $\theta$ irrational.
For $n \ge 1$, set
$$
s_n := s_n (\theta, \rho) = \big\lfloor n \theta + \rho \big\rfloor -
\big\lfloor (n-1) \theta + \rho \big\rfloor,
\quad
s'_n := s'_n (\theta, \rho) = \big\lceil n \theta + \rho \big\rceil -
\big\lceil (n-1) \theta + \rho \big\rceil. 
$$
Then, the infinite words
$$
{\bfs}_{\theta, \rho} := s_1 s_2 s_3 \ldots,
\quad
{\bfs}'_{\theta, \rho} := s'_1 s'_2 s'_3 \ldots 
$$
are, respectively, the lower and upper Sturmian words of slope $\theta$
and intercept $\rho$, written over the alphabet $\{0, 1\}$. 
For complex numbers $\alpha, \beta$ with $|\beta \alpha^\theta| < 1$, write
$$
\xi_{\s_{\theta,\rho}}(\beta, \alpha) = \sum_{n\ge 1} s_n \beta^{n} \alpha^{\sum_{h=1}^n s_h} 
= \sum_{n\ge 1} s_n \beta^{n} \alpha^{\lfloor n \theta + \rho \rfloor}. 
$$
Observe that, setting
$$
F_{\theta, \rho} (z_1, z_2) 
= \sum_{k_1 \ge 1} \, \sum_{k_2 = 1}^{\lfloor k_1 \theta + \rho \rfloor} \, z_1^{k_1} z_2^{k_2},
$$ 
we have 
$$
\xi_{\s_{\theta,\rho}}(\beta, \alpha) = (1 - \beta) F_{\theta, \rho} (\beta, \alpha) + \beta^2\alpha \lfloor \theta+ \rho \rfloor,    
$$
for any $\beta, \alpha$ satisfying $|\beta| < 1, |\beta \alpha^\theta | < 1$. 
The notation $\xi_{\s_{\theta,\rho}}$ was introduced in \cite{BuLa21} and we keep it 
in the present work.
%, instead of using the classical notation $F_{\theta, \rho}$. 
The transcendence of $F_{\theta, \rho} (\beta, \alpha)$ for nonzero algebraic numbers $\alpha, \beta$ has been 
widely studied, after the pioneering works of Mahler \cite{Mah29} and 
Loxton and van der Poorten \cite{LovdP77c} in the case $\rho=0$. 
Borwein and Borwein \cite[Theorem 0.4]{BoBo93} established that, if the slope $\theta$ has infinitely many partial quotients 
greater than or equal to $3$, then $\xi_{\s_{\theta,\rho}} (1/b, 1/a)$ is transcendental 
for every positive integers $a, b$ with $b \ge 2$. 
Komatsu \cite{Ko96b} (see also \cite{NiShTa92}) proved that, if the slope $\theta$ has unbounded partial quotients,
then $\xi_{\s_{\theta,\rho}}(\beta, \alpha)$ is transcendental for every nonzero complex algebraic numbers $\alpha, \beta$ with 
$|\beta \alpha^\theta | < 1$, under some technical condition. 
Lastly, Ferenczi and Mauduit \cite{FeMa97} used combinatorial properties of Sturmian sequences 
and Ridout's $p$-adic extension of Roth's theorem to show that $\xi_{\s_{\theta,\rho}}(1/b, 1)$ is 
transcendental for every integer $b \ge 2$. 

Our first main theorem is a considerable extension of all these results. 
%Maybe copy some parts of Section 2 of \cite{LaNo21} here 
%to explain the relationships between $\xi_{\s_{\theta,\rho}}(\beta,\alpha)$ and
%$F_{\theta, \rho}$, defined by
%$$
%F_{\theta, \rho} (z_1, z_2) 
%= \sum_{k_1 \ge 1} \, \sum_{k_2 = 1}^{\lfloor k_1 \theta + \rho \rfloor} \, z_1^{k_1} z_2^{k_2},
%$$
%where $z_1, z_2$ are in the domain $\{|z_1| < 1, |z_1| \cdot |z_2|^\theta < 1\}$. 

%\section{Results} \label{res} 

\begin{theorem} \label{maintrans}
Let $\theta$ and $\rho$ be real numbers with
$0 \le \theta, \rho < 1$ and $\theta$ irrational.
Let $\alpha, \beta$ be nonzero complex algebraic numbers such that $|\beta \alpha^\theta | < 1$ and $\beta \not= 1$. 
Then, the complex number 
$\xi_{\s_{\theta,\rho}}(\beta, \alpha)$ is transcendental. In particular, if $|\beta| < 1$, then 
the complex numbers 
$$
h_{\theta, \rho} (\beta), \quad F_{\theta, \rho} (\beta, \alpha).   
$$
are transcendental.
\end{theorem}

Since, for every $\alpha$ in the open unit disc, we have
$$
\xi_{\s_{\theta,\rho}}(1, \alpha) = \frac{\alpha}{1 - \alpha},
$$
the assumption $\beta \not= 1$ in Theorem \ref{maintrans} is necessary. 

When the slope $\theta$ has unbounded partial quotients in its continued fraction expansion, 
Theorem \ref{maintrans} was proved by Komatsu \cite{Ko96b}, under some mild additional assumption on $\alpha$ 
and $\beta$. For the sake of completeness, we display a complete proof in Section \ref{unbounded}. 

Adamczewski and Bugeaud \cite[Proposition 11.1]{AdBu11} proved that the Diophantine exponent 
(which measures the repetitions occurring at the beginning or near the beginning of an    
infinite word, see \cite[p. 70]{AdBu11}) of a Sturmian 
sequence is infinite if and only if its slope $\theta$ has unbounded partial quotients, independently of the 
value of its intercept $\rho$. Under this assumption, the $p$-adic 
Schmidt Subspace Theorem applies to show that, for every $\rho$ in 
$[0, 1)$ and every nonzero algebraic number $\beta$ in the open unit disc, the complex number 
$\xi_{\s_{\theta,\rho}}(\beta, 1)$ is either transcendental, or lies in $\Q (\beta)$; see \cite[Theorem 1]{AdBu07b}. 
A different, more involved, application of the $p$-adic Schmidt Subspace Theorem allows us to get    
the same conclusion if $\theta$ has bounded partial quotients; details will be given in a subsequent paper.

The proof of Theorem \ref{maintrans} follows Mahler's method and its extension 
%to chains of functional equations 
by Loxton and van der Poorten \cite{LovdP77c}. 
A key point is the following construction, leading to a chain of functional 
equations. For an irrational real number $\theta$ in $(0, 1)$, write 
$$
\theta = [0 ; a_1, a_2, \ldots], \quad 
\theta_k = [0; a_{k+1}, a_{k+2}, \ldots ], \quad k \ge 0,   
$$
in such a way that 
$$
\theta_0 = \theta,  \quad 
\theta_{k+1} = {1\over \theta_k}  - a_{k+1} =\left\{ {1\over \theta_k}\right\},  \quad k \ge 0,  
$$
where $\{ \cdot \}$ denotes the fractional part function.  
Let $(p_k / q_k)_{k \ge 0}$ denote the sequence of convergents to $\theta$. 
An elementary calculation yields the equation
$$
F_{{\theta_k},0} (z_1, z_2) = - F_{{\theta_{k+1}, 0}} (z_1^{a_{k+1}} z_2, z_1) + 
{z_1^{a_{k+1}+1} z_2 \over (1 - z_1^{a_{k+1}} z_2) (1 - z_1) }.   
$$
When $\theta$ is a quadratic irrational, the sequence $(\theta_k)_{k \ge 1}$ 
is ultimately periodic, and this chain of functional equation yields a single functional equation.
Namely, assuming that $\theta_{k+s} = \theta_k$ for $k \ge 0$  
and an even positive integer $s$, we end up with 
a functional equation of the form 
$$
F_{\theta, 0} (z_1, z_2)  = 
F_{\theta, 0} (z_1^{q_s}  z_2^{p_s}, z_1^{q_{s-1}} z_2^{p_{s-1}})  
+ R(z_1, z_2),                               
$$
where $R(z_1, z_2)$ is in $\Q(z_1, z_2)$ and which has
been treated by Mahler \cite{Mah29}. 
In general, we have  a system of functional equations
$$
F_{\theta, 0} (z_1, z_2) = 
(-1)^{k} F_{{\theta_{k}, 0}} (z_1^{q_{k}}  z_2^{p_{k}}, z_1^{q_{k-1}} z_2^{p_{k-1}})                       
+ R_k (z_1, z_2),  \quad k\ge 1.
$$
Loxton and van der Poorten \cite{LovdP77c}  developed a general theory which applies to such chains of equations 
under some technical constraints. These assumptions may be  satisfied (for a suitable subsequence of the indices $k$) 
if we assume that  the sequence $(a_k)_{k\ge 1}$   is bounded. 
By means of our new result on the structure of Sturmian sequences \cite{BuLa21} (see Proposition \ref{combsturm} below),  
we are able to show that this approach also works for the more general 
series $F_{\theta, \rho} (z_1, z_2)$. 
The unbounded case, treated in Section \ref{unbounded},  is related to the second part of our paper 
devoted to continued fractions expansions.

We stress an immediate consequence of Theorem \ref{maintrans}. 
For more on $\beta$-expansions of real numbers, the reader is directed to \cite{AdBu07b} and the references given therein. 

\begin{corollary}
Let $\alpha$ and $\beta$ be real algebraic numbers with $\beta > 1$. 
Then, the $\beta$-expansion of $\alpha$ is not given by a Sturmian sequence. 
\end{corollary}

Theorem \ref{maintrans} asserts that any power series whose sequence of coefficients is 
a Sturmian sequence of integers sends non-zero algebraic points in the unit disc to transcendental points. 
This is not the case for every automatic series, as shown by Adamczewski and Faverjon \cite[Section 8.1]{AdFa17}, 
who gave the example of an automatic series taking an 
algebraic value at any point of the form $\phi^{1 / 3^\ell}$, where $\phi = (1 - \sqrt{5})/2$
and $\ell \ge 1$.

Let $a$ and $b$ be positive integers with $b\ge2$. By Theorem \ref{maintrans}, the real numbers $\xi_{\s_{\theta,\rho}}(1/b, 1/a)$
and $\xi_{\s'_{\theta,\rho}}(1/b, 1/a)$ are transcendental. 
We now deal with the continued fraction expansion of the real numbers $\xi$ of the form
$$
(b-1)\xi_{\s_{\theta,\rho}}(1/b, 1/a) \quad {\rm or}\quad  (b-1)\xi_{\s'_{\theta,\rho}}(1/b, 1/a).
$$
When $a=1$, we will recover the expansion of Sturmian numbers obtained in \cite{BuLa21}.

We denote by  $(b_k)_{k \ge 1}$  the sequence of digits of the number
\begin{equation}\label{defbk}
\rho -\theta = \sum_{k\ge 0} b_{k+1}(q_k\theta-p_k), 
\end{equation}
written in the Ostrowski numeration system 
with base $\theta$ (normalized as in Theorem 2.1 of \cite{BuLa21} or in Theorem 4.2 when  $\rho$ is of the form $-m \theta + p$, with $m, p$ nonnegative integers). We set 
(by convention, an empty sum is equal to zero)
\begin{equation}\label{defrktk}
t_k = \sum_{ j=1}^{k} b_j q_{j-1} ,\, \, \tit_k =  \sum_{ j=1}^{k} b_j p_{j-1}, \quad 
%{\rm and} \quad 
r_k =q_k-t_k, \,\,\tir_k =p_k - \tit_k, \quad k\ge 0.
\end{equation}
For $k \ge 0$, set 
$$
\begin{aligned}
c_k = &
\begin{cases}
 {b^{a_1-b_1}a-b\over b-1},   \quad  & \text{when} \quad  k= 0,
\\
{b^{r_{k} + q_{k-1}}a^{\tir_{k}+p_{k-1} }
\left(\left(b^{q_k}a^{p_k}\right)^{a_{k+1} - b_{k+1} - 1}  - 1 \right)\over b^{q_{k}}a^{p_k} - 1},
 \quad  & \text{when} \quad  k\ge 1, 
 \end{cases}
\\
d_k =& b^{t_{k}}a^{\tit_k} - 1,
\\
e_k =& b^{r_{k}} a^{\tir_k}- 1, 
\\
f_k = & b^{t_{k}} a^{\tit_k}  {(b^{q_{k}}a^{p_k})^{b_{k+1}} - 1 \over b^{q_{k}}a^{p_k} - 1}.
\end{aligned}
$$
When $a=1$, the four sequences $(c_k)_{k\ge 0}, (d_k)_{k\ge 0}, (e_k)_{k\ge 0}, (f_k)_{k\ge 0}$ 
coincide with the corresponding ones introduced in \cite{BuLa21}. 
We point out that some elements of these  sequences may be non-positive, exactly in the same situations as in \cite{BuLa21}. 
For example, $f_k$ is equal to $0$ when $b_{k+1} = 0$ 
and $c_{k+1}$ is equal to $0$ when $a_{k+2} = b_{k+2} + 1$. In the case where $a_{k+2} = b_{k+2}$, we have $b_{k+1} = 0$, thus 
$r_{k} + q_{k+1} = r_{k+1} + q_{k}, \tir_k +p_{k+1}= \tir_{k+1}+p_k$, 
so that
$$
\eqalign{
c_{k+1} & = b^{r_{k+1} + q_{k}}a^{\tir_{k+1}+p_k} \, {(b^{q_{k+1}} a^{p_{k+1}})^{-1}- 1 \over b^{q_{k+1}}a^{p_{k+1}} - 1} 
\cr 
& =
{b^{r_{k}}a^{\tir_k} - b^{r_{k} + q_{k+1}}a^{\tir_k+p_{k+1}} \over b^{q_{k+1}}a^{p_{k+1}} - 1 } = - b^{r_{k}}a^{\tir_k} = - e_{k} - 1 
\cr}
$$
is  negative. Notice as well  that $e_k$ is always positive, because $b\ge 2, a\ge 1, r_k \ge 1$,  and that  $d_k$ is non-negative and vanishes if and only if $t_{k}=\tit_k=0$, that is to say when  $b_1=\cdots = b_{k}=0$.

Keeping this in mind, and with some abuse of language, the next theorem asserts that 
$$
[0 ; c_0, d_0, 1, e_0, f_0, c_1, d_1, 1, e_1, f_1, c_2, \ldots ]
$$
is an (improper) continued fraction expansion of $\xi$. 
In order to rule out non-positive elements in the sequence 
$$
c_0, d_0, 1, e_0, f_0, c_1, d_1, 1, e_1, f_1, c_2, \ldots 
$$
we apply to it some contraction rules. 
The precise statement is as follows.

\begin{theorem}  \label{fraccont}
 Let $a$ and $b$ be positive integers. 
Assume  that $b \ge 2 $ and that $a$ is conguent to $1$ modulo $b-1$. 
Let $A_1, A_2, A_3, \ldots$ be the sequence of positive integers obtained from the  sequence 
$c_0, d_0, 1$, $e_0, f_0, c_1, d_1, 1, e_1, f_1, c_2, \ldots$ after the application of the following rules:
\\
$(i)$ For any $k\ge 0$ such that $a_{k+2}=b_{k+2}$, replace the string of the 9 consecutive terms
$$
c_k, d_k, 1, e_k, f_k=0, c_{k+1}= -e_k-1,d_{k+1}=d_k,1, e_{k+1} 
$$ 
by the single element
$ c_{k} + e_{k+1} + 1 $.
\\
$(ii)$ Replace any three consecutive elements of this new sequence of the form $x, 0, y$ by the integer $x+y$ ($x$ and $y$ may vanish) and  continue the reduction until one obtains positive integers.
\\
Then,  the  continued fraction expansion of $\xi$ is given by
$$
\xi = [0 ; A_1, A_2, A_3, \ldots ]. 
$$
\end{theorem}

Observe that the sequence $(A_j)_{j \ge 1}$ is well-defined. Indeed, $c_k$ and $c_{k+1}$ 
cannot be both negative, since we cannot have simultaneously $a_{k+1} = b_{k+1}$ and $a_{k+2} = b_{k+2}$ by Ostrowski numeration rules. The process $(ii)$ enables us to get rid of the $0$ after having ruled out the negative terms using  $(i)$.  The  occurrences of  $0$, after performing  the rule $(i)$,    are fully described thanks to  the six cases  displayed in Section 7 of \cite{BuLa21}, which remain unchanged in our setting. 
For convenience, we reproduce the list below.

$(ii)_1$  $\, b_k=0$ and $a_{k+2}=b_{k+2}$ with $k\ge 1$, corresponding to the string 
$$
1, e_{k-1}, f_{k-1}=0, c_k+e_{k+1}+1,f_{k+1}, \quad {\rm  where} \quad e_{k-1} >0 , \,  c_k+e_{k+1}+1 >0 , \, f_{k+1}>0.
$$

$(ii)_2$  $\, b_{k+1}=0$,  $ t_{k+1} \ge 1$ and $a_{k+2} \ge b_{k+2}+2 $ with $k\ge 0$, corresponding to the string 
$$
 1, e_{k}, f_{k}=0, c_{k+1}, d_{k+1}, \quad {\rm  where} \quad e_{k} >0 , \,  c_{k+1}>0, \, d_{k+1} >0.
$$

$(ii)_3 $  $\, b_{k+1} \ge 1$  and $a_{k+2} = b_{k+2}+1 $ with $k\ge 0$, corresponding to the string  
$$
e_k, f_{k}, c_{k+1}=0, d_{k+1},1, \quad {\rm  where} \quad  e_k >0, \, f_k>0, \, d_{k+1} >0.
$$

$(ii)_4 $  $\, b_{k+1} =0 $, $t_{k+1} \ge 1$  and $a_{k+2} = b_{k+2}+1 $ with $k\ge 0$, corresponding to the string  
$$
1,e_k, f_{k}=0, c_{k+1}=0, d_{k+1},1, \quad {\rm  where} \quad  e_k >0, \, d_{k+1} >0.
$$

$(ii)_5 $  $\, t_{k+1} =0 $   and $a_{k+2} \ge  b_{k+2}+2 $ with $k\ge 0$, corresponding to the string  
$$
1,e_k, f_{k}=0, c_{k+1}, d_{k+1}=0,1,e_{k+1}, \quad {\rm  where} \quad  e_k >0, \,c_{k+1} >0 , \,  e_{k+1} >0.
$$

$(ii)_6 $  $\, t_{k+1} =0 $   and $a_{k+2} =  b_{k+2}+1 $ with $k\ge 0$, corresponding to the string  
$$
1,e_k, f_{k}=0, c_{k+1}=0, d_{k+1}=0,1,e_{k+1}, \quad {\rm  where} \quad  e_k >0, \,  e_{k+1} >0.
$$

\medskip 
As a simple  example, we obtain  the
\begin{corollary}
  Assume that  $a_k - b_k \ge 2$ and $b_k \ge 1$ for every $k \ge 1$.     
Then, the continued fraction expansion of $\xi$ is given by  
   $$
\xi = [0 ; c_0 + 1, e_0, f_0, c_1, d_1, 1, e_1, f_1, c_2, \ldots ].     
$$
\end{corollary}

\begin{proof} Observe that $d_0= 0$, while all the other elements of the sequence 
$$
c_0, d_0, 1, e_0, f_0, c_1, d_1, 1, e_1, f_1, c_2, \ldots 
$$
 are positive. 
 \end{proof}

When $a$ and $b$ are positive integers with $b \ge 2$ and $a$ not congruent to $1$ modulo $b-1$, we get the 
regular continued fraction expansion of $1/\xi - (c_0 + 1) = 1/\xi - (b^{a_1 - b_1} a - 1)/(b-1)$.

As a consequence of Theorem \ref{fraccont}, we obtain an expression for the 
irrationality exponent of any real number $\xi$ as above in terms of its slope and its intercept. 

%We can deduce from Theorem 2.3 the irrationality exponent of $\xi_b (\theta, \rho)$. 
Keep our notation and define
$$
\nu_k(1) = 2+ {t_k\over r_{k+1}}, \quad \nu_k(2)
= 2+ {r_{k} \over r_{k+1}+ t_k}, 
$$
$$
\nu_k(3) = 1+ { q_{k+1}\over r_{k+1}+ q_k},
\quad \nu_k(4)= 1+ { r_{k+2}\over q_{k+1}}.
$$
Put 
$$
\eqalign{
\nu(1) & = \limsup_{k \to + \infty} \, \{ \nu_k(1) \, : \, a_{k+1}-b_{k+1} \ge 1 
\hbox{ and $a_{k+2}-b_{k+2} \ge 1$}\},
\cr
\nu(2)  &= \limsup_{k \to + \infty} \, \{ \nu_k(2) \, : \, a_{k+2}-b_{k+2} \ge 1 \},
\cr}
$$
and, for $j = 3, 4$, 
$$
\nu(j) = \limsup_{k \to + \infty}  \nu_k(j).
$$

\begin{theorem}\label{expirrat}
Let $b \ge 2$ and $a \ge 1$ be integers. 
The irrationality exponent of $\xi_{\s_{\theta,\rho}}(1/b, 1/a)$ (resp., of $\xi_{\s_{\theta,\rho}}(1/b, 1/a)$) 
is equal to
$$
\max\{ \nu (1), \nu (2), \nu (3), \nu (4) \}. 
$$
\end{theorem}

Theorem \ref{expirrat} extends \cite[Theorem 2.4]{BuLa21} which covers the case $a=1$.

\section{Sturmian words}

We collect in this Section some important properties of the Sturmian words $\s_{\theta,\rho}$ and $\s'_{\theta,\rho}$,  obtained in \cite{BuLa21}. Recall 
that $(p_k / q_k)_{k \ge 0}$ is the sequence of convergents to $\theta = [0; a_1, a_2, \ldots ]$ and that    
the sequences $(b_k)_{k \ge 1}$, $(r_k)_{k \ge 0}$, $(t_k)_{k \ge 0}$, $(\tir_k)_{k \ge 0}$, $(\tit_k)_{k \ge 0}$ are defined  
in \eqref{defbk} and \eqref{defrktk}.

\begin{lemma} \label{recrk}
We have 
$$
r_0= 1,\quad  \tir_0 = 0,\quad 
r_1 = a_1-b_1, \quad \tir_1= 1, 
$$
 and the following recursion formulae hold for any $k\ge 0$: 
$$
\eqalign{
r_{k+1} & = r_{k}+ (a_{k+1}-b_{k+1}-1)q_k+ q_{k-1},   
\cr
\tir_{k+1}& = \tir_k+ (a_{k+1}-b_{k+1}-1)p_k+ p_{k-1}.
\cr}
$$
It follows that 
$$
\eqalign{
r_{k+1} & = 1-q_k+ \sum_{j= 0}^{k} (a_{j+1}-b_{j+1})q_j= q_{k+1}-t_{k+1},   \quad k \ge 0, 
\cr
\tir_{k+1} & = 1-p_k+ \sum_{j= 0}^{k} (a_{j+1}-b_{j+1})p_j= p_{k+1}-\tit_{k+1},   \quad k \ge 0. 
\cr}
$$
Moreover, we have $0\le t_k < q_k$,  $0 \le \tit_k \le p_k$, $1 \le r_k \le q_k$, and $0 \le \tir_k \le p_k$, for every $k\ge 0$. 
\end{lemma}

\begin{proof} 
Notice that the classical recurrence relations $q_{j+1}= a_{j+1}q_j +q_{j-1}$ and $p_{j+1}= a_{j+1}p_j +p_{j-1}$ for any $j\ge 0$, arising 
from the theory of continued fractions,  yield the formulae
\begin{equation} \label{sommeqk}
 \sum_{j=0}^k a_{j+1}q_j= a_1 + \sum_{j=1}^k q_{j+1}-q_{j-1} = q_{k+1}+ q_k-1 , \quad k\ge 0,
\end{equation}
 and
\begin{equation} \label{sommepk}
 \sum_{j=0}^k a_{j+1}p_j=  \sum_{j=1}^k p_{j+1}-p_{j-1} = p_{k+1}+ p_k-1    , \quad k\ge 0.
\end{equation}
 Then, we deduce the formulae of Lemma \ref{recrk} from \eqref{defrktk}, \eqref{sommeqk} and \eqref{sommepk}. 
Notice finally that the Ostrowski numeration rules ($0\le b_1\le a_1-1$, $0\le b_k\le a_k$, for $k\ge 1$, and $b_{k+1}= a_{k+1}$
implies $b_k=0$, for every $k \ge 1$) yield by induction on $k$ the required inequalities.
\end{proof}

Let
$$
\bc_\theta : = \s_{\theta, \theta}= \s'_{\theta,\theta}, 
$$
be the characteristic word of slope $\theta$. For $k\ge 1$, we denote by $M_k$ the prefix of length $q_k$ of $\bc_\theta$.
Set $M_0 = 0$ and $M_{-1}=1$.

\begin{proposition}\label{combsturm}
Define inductively two sequences of finite words $(T_k)_{k\ge 0}$ and $(R_k)_{k\ge 0}$ 
on $\{0, 1\}$ by letting $T_0$ be the empty word, $R_0=0$, and by  the recursion formulae
\begin{equation}\label{recTk}
T_{k+1}= M_k^{b_{k+1}}T_k 
\end{equation}
and
\begin{equation}\label{recRk}
R_{k+1}= \begin{cases} R_k M_k^{a_{k+1}-b_{k+1}-1}M_{k-1} & \text{if} \quad b_{k+1}< a_{k+1}, 
\\
R_{k-1} & \text{if}\quad  b_{k+1}=a_{k+1},
\end{cases}
\end{equation}
for any $k\ge 0$. Then, $T_k$ (resp. $R_k$) has length $t_k$ (resp. $r_k$) and contains $\tit_k$ (resp. $\tir_k$) letters $1$. Set 
$$
V_k = R_kT_k, \quad k\ge 0.
$$
The word $V_k$ has length $q_k$, contains $p_k$ letters $1$, 
and its first $q_k-1$  letters coincide with those  of $\s_{\theta,\rho}$ (or $\s'_{\theta,\rho}$). Moreover $M_k= T_kR_k$ and the sequence $(V_k)_{k\ge 0}$ satisfies the recurrence relations
$$
V_{-1}=1,\quad  V_0 = 0, \quad V_1 = V_0^{a_1-b_1-1}V_{-1}V_0^{b_1},
\quad V_{k+1}= V_k^{a_{k+1}-b_{k+1}}V_{k-1}V_k^{b_{k+1}}, \, k\ge 1.
$$
\end{proposition}

\begin{proof}
Proposition \ref{combsturm} is a reformulation of the results of \cite[Section 3]{BuLa21}, with the  exception of the assertions concerning the number of letters $1$. 
These  follow from Lemma \ref{recrk}, combined with the  recursion formulae \eqref{recTk} and \eqref{recRk} established in  
\cite[Lemma 3.3]{BuLa21}, by observing that the word $M_k$ contains $p_k$ letters $1$.  
Notice that when $a_{k+1} = b_{k+1}$, we have $b_k=0$, so that 
$$
\begin{aligned}
\tir_k & =  \tir_{k-1}+(a_{k} -1)p_{k -1}+ p_{k-2}= \tir_{k-1}+ p_k -p_{k-1}, 
\\
\tir_{k+1} &= \tir_{k}+(a_{k+1}-b_{k+1}-1)p_k + p_{k-1} =\tir_k -p_k+ p_{k-1}= \tir_{k-1}.
\end{aligned}
$$
It follows that  $R_{k+1}=R_{k-1}$ contains $\tir_{k+1}= \tir_{k-1}$ letters $1$, as claimed. 
\end{proof}

\begin{definition} \label{formalint}
The sequence $(b_k)_{k \ge 1}$ is called the formal intercept of the Sturmian word 
$\s_{\theta,\rho}$ of slope $\theta$ and intercept $\rho$.
\end{definition}

The next lemma will be used in Sections \ref{cfe} and \ref{unbounded}.    

\begin{lemma}  \label{rktirk}
As $k$ tends to infinity, we have
$$
r_k \theta - \tir_k = O(1), \quad t_k \theta - \tit_k = O(1).
$$
\end{lemma}

\begin{proof}
Lemma \ref{recrk} yields the formula
 $$
 \eqalign{
r_k\theta -\tir_k & = q_{k}\theta-p_{k} -\sum_{j=0}^{k-1}b_{j+1}(q_j\theta -p_j)  \cr
& = \theta -\rho +(q_{k}\theta-p_{k}) +\sum_{j\ge k}b_{j+1}(q_j\theta -p_j),
}
$$
recalling the Ostrowski expansion 
$$
\rho -\theta = \sum_{j\ge 0}b_{j+1}(q_j\theta -p_j). 
$$
This shows that $r_k \theta - \tir_k = O(1)$. Since $|q_k \theta - p_k| \le 1$, we get the 
second estimate. 
\end{proof}

\section{Continued fraction expansion }    \label{cfe} 

The main goal of this Section is to prove Theorem \ref{fraccont},  and to give further results on the convergents of $\xi$.

 For a finite word $W=w_1\dots w_\ell$ over the alphabet $\{0, 1\}$ and variables $a, b$, set     
$$
W(b,a) = \sum_{n=1}^\ell w_nb^{\ell-n}a^{\sum_{h=n+1}^\ell w_h} = 
b^\ell a^{\sum_{h=1}^\ell w_h} \sum_{n=1}^\ell w_n 
\left({1\over b}\right)^n\left( {1\over a}\right)^{\sum_{h=1}^n w_h}.
$$
Note that the exponent $\sum_{h=1}^n w_h$ counts the number of letters $1$ in the prefix of length $n$ of the word $W$. 

Now, if $\bx =x_1x_2\dots$ is an infinite word over the alphabet $\{0, 1\}$, recall that we have set
$$
\xi_{\bx}(\beta, \alpha) = \sum_{n\ge 1} x_n \beta^{n}\alpha^{\sum_{h=1}^nx_h}.
$$
When $\bx$ is  an ultimately   periodic word, $\xi_\bx(\beta, \alpha)$ is a rational function in the two variables $\beta$ and $\alpha$. Set $a= 1/\alpha$ and $b= 1/\beta$. More precisely, we have the 

\begin{lemma} \label{numerateur}
Let $Y=y_1\dots y_r$ and $Z= z_1\dots z_s$ be two finite words over $ \{0, 1\}$.  
Put $\tir = y_1+ \cdots + y_r$ and $\tis = z_1+ \cdots + z_s$. Then
$$
YZ(b,a)= b^sa^\tis Y(b,a)+ Z(b,a),
$$
where $YZ= y_1\dots y_r z_1 \dots z_s$ stands for the concatenation of the two words $Y$ and $Z$. 
Moreover, if $| b^s  a^\tis | > 1$, then  
$$
\xi_{Z^\infty}\left(\beta,\alpha\right) = { Z(b,a)\over b^s a^\tis -1}
\quad {\rm and} \quad 
\xi_{YZ^\infty}\left(\beta,\alpha \right) = { YZ(b,a)-Y(b,a)\over b^ra^\tir(b^s a^\tis -1)}, 
$$
where $Z^\infty$ stands for the concatenation of infinitely many copies of $Z$.     
\end{lemma}

\begin{proof}The first formula immediately follows from the definition.

By setting  $\x = YZ^\infty$ and writing $n= r+js+m$ for $n\ge r+1$, we obtain by periodicity
$$
\eqalign{
\xi_{\x}(\beta, \alpha) &= \sum_{n\ge 1} x_n\beta^n\alpha^{\sum_{h= 1}^nx_h} \cr
& = \sum_{n= 1}^r y_n\beta^n\alpha^{\sum_{h= 1}^ny_h}  
+  \sum_{j\ge 0}\sum_{m=1}^s z_m \beta^{r+ js+ m} \alpha^{\tir + j\tis + \sum_{h= 1}^m z_h}
\cr
 & =  \sum_{n= 1}^r y_n\beta^n\alpha^{\sum_{h= 1}^ny_h}  
 + { \beta^{r}\alpha^{\tir}\sum_{m=1}^s z_m\beta^m\alpha^{\sum_{h=1}^mz_h}\over 1 -\beta^s\alpha^\tis}
 \cr
 & = {Y(b,a)\over b^ra^\tir} + {Z(b,a)\over b^ra^\tir(b^sa^\tis -1)} \cr
 & =
 {Y(b,a)(b^sa^\tis-1) + Z(b,a)\over b^ra^\tir(b^sa^\tis-1)}
 ={YZ(b,a)-Y(b,a)\over b^ra^\tir(b^sa^\tis-1)}.
 \cr}
$$
When $Y$ is the empty word, we obtain the  formula $\xi_{Z^\infty}\left( \beta,\alpha\right) = { Z(b,a)\over b^s a^\tis -1}$. 
\end{proof}

We use Lemma \ref{numerateur} in order to construct  rational fractions in $a$ and $b$ 
associated to four sequences of perodic words which approach the Sturmian word $\s = \s_{\theta,\rho}$. For any $k\ge 0$, define
$$
(1)_k =   { (b-1) (R_{k+1} (b,a) -R_k  (b,a) ) \over 
b^{r_k} a^{\tir_{k}} ( b^{r_{k+1}-r_k} a^{\tir_{k+1}-\tir_k }-1)},
$$
which is associated to the word $R_k (M_k^{a_{k+1}-b_{k+1}-1} M_{k-1})^{\infty}$ whenever $a_{k+1}-b_{k+1}\ge 1$. Next, set
$$
(2)_k =   {(b-1) (R_{k+1}T_k)  (b,a) \over  b^{r_{k+1}+ t_k} a^{ \tir_{k+1} + \tit_{k} } -1},
$$
associated to the purely periodic word $(R_{k+1}T_k)^\infty$. The third approximation is 
$$
(3)_k =   {(b-1) \big((R_{k+1}M_k)  (b,a)  -R_{k+1}  (b,a) \big)\over  
b^{r_{k+1}} a^{ \tir_{k+1}} (b^{q_k} a^{p_k}-1)},
$$
associated to the word $R_{k+1}M_k^\infty$. Put finally
$$
(4)_k =   { (b-1)V_{k+1}  (b,a) \over b^{q_{k+1}}  a^{p_{k+1}} -1},     
$$
associated to the purely periodic word $V_{k+1}^\infty=(R_{k+1}T_{k+1})^\infty$. 

\medskip

We now give  an  analogue of \cite[Lemma 7.1]{BuLa21} in our framework.  We use 
the notation ${P\over Q}= c\cdot  {P' \over Q'} \oplus {P''\over Q''}$ between fractions to mean  that 
both relations  $P= cP'+P''$ and  $Q= cQ'+Q''$ hold true. 
 Similarly, $(2)_k \ominus (1)_k$ stands below for the fraction whose numerator (resp. denominator)  is  the difference between the numerators (resp. denominators) of $(2)_k$ and $(1)_k$.  It is convenient to define formally $(3)_{-1} = {b-1\over 0}$ and $(4)_{-1}= {0\over b-1}$. Then, we have the 

\begin{lemma} \label{rec(j)k}
For any $k \ge 0$, we have the following relations: 
$$
(1)_{k} = c_k \cdot (4)_{k-1} \oplus  (3)_{k-1},  
$$
$$
(2)_{k} \ominus (1)_k = d_k   \cdot (1)_{k} \oplus  (4)_{k-1},   
$$
%$$
%d_k  \cdot (1)_{k+1} = ( (2)_{k+1} - (1)_{k+1} ) - (4)_k, 
%$$
$$
(2)_{k} = 1 \cdot \bigl( (2)_{k} \ominus (1)_{k} \bigr) \oplus   (1)_{k},
$$
%$$
%(e_k +1) \cdot (2)_{k+1} = (3)_{k+1} + (1)_{k+1},
%$$
$$
(3)_{k} = e_k  \cdot (2)_{k} \oplus  \bigl( (2)_{k} \ominus (1)_{k} \bigr), 
$$
$$
(4)_{k} = f_k \cdot (3)_{k} \oplus  (2)_{k}. 
$$
\end{lemma}

\begin{proof}
We compute 
$$
\displaylines{
(1)_0= {b-1\over b^{a_1-b_1}a-b}, \quad (2)_0  \ominus (1)_0 = {0\over b-1} ,\quad (2)_0 = {b-1 \over b^{a_1-b_1}a-1}, 
\cr
\quad (3)_0 = {(b-1)^2\over b^{a_1-b_1}a(b-1)}, \quad (4)_0 = { b^{b_1}(b-1)\over b^{a_1}a -1}.
}
$$
We have 
$$
c_0= {b^{a_1-b_1}a- b\over b-1}, \quad d_0= 0, \quad e_0= b-1,\quad f_0= {b^{b_1}-1\over b-1},
$$
so that the five above relations are verified for $k=0$. 

Assume now that $k\ge 1$. The third relation is obvious.  Let us check the four remaining  relations. The denominators of 
$(1)_k, (2)_k \ominus (1)_k, (2)_k, (3)_k, (4)_k$ are respectively
$$
\displaylines{
Q_{(1)_k}= b^{r_{k+1}}a^{\tir_{k+1}} -b^{r_k}a^{\tir_k}, \quad Q_{(2)_k \ominus (1)_k}= b^{r_{k+1}+t_k}a^{\tir_{k+1}+\tit_k}-1-(b^{r_{k+1}}a^{\tir_{k+1}} -b^{r_k}a^{\tir_k}), 
\cr
Q_{(2)_k} = b^{r_{k+1}+t_k}a^{\tir_{k+1}+\tit_k}-1,\quad Q_{(3)_k}= b^{r_{k+1}}a^{\tir_{k+1}}( b^{q_k}a^{p_k}-1), \quad  Q_{(4)_k} = b^{q_{k+1}}a^{p_{k+1}}-1.
}
$$
Using Lemma \ref{recrk}, we check that 
$$
\eqalign{
 {Q_{(1)_k}- Q_{(3)_{k-1}}\over Q_{(4)_{k-1}}}= &{ b^{r_{k+1}}a^{\tir_{k+1}} - b^{r_k}a^{\tir_k} - b^{r_k+ q_{k-1}}a^{\tir_k +p_{k-1} }+b^{r_k}a^{\tir_k}\over b^{q_k}a^{p_k}-1}
\cr
=  & {b^{r_k+ q_{k-1}}a^{\tir_k +p_{k-1} } ( (b^{q_k}a^{p_k})^{a_{k+1}-b_{k+1}-1}-1)\over b^{q_k}a^{p_k}-1} =c_k,
\cr}
$$
as required. Similarly,  we have to check that $d_k= {Q_{(2)_k}-Q_{(1)_k} -Q_{(4)_{k-1}}\over Q_{(1)_k}}$, or equivalently $d_{k}+1= {Q_{(2)_k} -Q_{(4)_{k-1}}\over Q_{(1)_k}}$. Now,
$$
{Q_{(2)_k} -Q_{(4)_{k-1}}\over Q_{(1)_k}}= { b^{r_{k+1}+t_k}a^{\tir_{k+1}+\tit_{k}} -1 -(b^{q_k}a^{p_k}-1)\over b^{r_{k+1}}a^{\tir_{k+1}} - b^{r_k}a^{\tir_k}} = b^{t_k}a^{\tit_k} = d_k+1,
$$
since $q_k= r_k+t_k$ and $p_k= \tir_k+\tit_k$. For the fourth relation, we have to show that $e_k= {Q_{(3)_k}-Q_{(2)_k} + Q_{(1)_k}\over Q_{(2)_k}}$, or equivalently that  $e_k+1= {Q_{(3)_k} + Q_{(1)_k}\over Q_{(2)_k}}$. But
$$
{Q_{(3)_k} + Q_{(1)_k}\over Q_{(2)_k}}= { b^{r_{k+1}+q_k}a^{\tir_{k+1}+ p_k} - b^{r_k}a^{\tir_k} \over b^{r_{k+1} +t_k}a^{\tir_{k+1}+ \tit_k} -1} = b^{r_k}a^{\tir_k} = e_k+1,
$$
writing again $q_k= r_k+t_k$ and $p_k= \tir_k+\tit_k$.  For the fifth relation, we compute
\begin{align*}
{Q_{(4)_k} -Q_{(2)_k}\over Q_{(3)_k}} 
& ={ b^{q_{k+1}}a^{p_{k+1}}- b^{r_{k+1}+t_k}a^{\tir_{k+1}+\tit_k}\over b^{r_{k+1}}a^{\tir_{k+1}}(b^{q_k}a^{p_k}-1)} \\
& = { b^{t_k + b_{k+1}q_k} a^{\tit_k+ b_{k+1}p_k} -b^{t_k}a^{\tit_k}\over b^{q_k}a^{p_k}-1} =f_k,
\end{align*}
since 
$$
q_{k+1}= r_{k+1}+t_{k+1}= r_{k+1}+ t_k+b_{k+1}q_k
$$
and 
$$
p_{k+1}= \tir_{k+1}+\tit_{k+1}= \tir_{k+1}+ \tit_k+b_{k+1}p_k.
$$
It remains to deal with the numerators. For the first relation, we have to show that
$$
c_kV_k(b,a)+ R_kM_{k-1}(b,a) -R_k(b,a) = R_{k+1}(b,a)-R_k(b,a).
$$
Assume first that $a_{k+1}-b_{k+1}\ge 1$. Using  Lemma \ref{numerateur}, Proposition \ref{combsturm} and noting that $(R_{k} T_{k})^{a_{k+1} - b_{k+1}  - 1} R_{k} M_{k-1}= R_{k+1}$ by \eqref{recRk}, we compute
$$
\eqalign{ 
c_kV_k(b,a)& = V_k(b,a) \times b^{r_{k} + q_{k-1}} a^{\tir_k+ p_k}
{(b^{q_k}a^{p_k})^{a_{k+1} - b_{k+1} - 1  } - 1 \over b^{q_{k}}a^{p_k} - 1} \cr
&= R_k T_k(b,a) \times b^{r_k+q_{k-1}}a^{\tir_k+p_{k-1}}(1+ b^{q_k}a^{p_k}+ \cdots + (b^{q_k}a^{p_k})^{a_{k+1} - b_{k+1} - 2 }  )
 \cr
& =  b^{r_k+q_{k-1}}a^{\tir_k+p_{k-1}}(R_{k} T_{k})^{a_{k+1} - b_{k+1} - 1} (b,a) \cr
& = (R_{k} T_{k})^{a_{k+1} - b_{k+1}  - 1} R_{k} M_{k-1}(b,a) - R_{k} M_{k-1}(b,a)  \cr
& = R_{k+1}(b,a)  - R_{k} M_{k-1}(b,a), \cr
}
$$
as required.
Assume finally that $a_{k+1} = b_{k+1}$. Then, $c_k = - b^{r_{k-1}}a^{\tir_{k-1}}$ and we  have 
$$
\eqalign{
V_k(b,a) \times ( - b^{r_{k-1}}a^{\tir_{k-1}}) & = - V_k R_{k-1}(b,a) + R_{k-1}(b,a) \cr
& = - R_k T_k R_{k-1}(b,a) + R_{k-1}(b,a) \cr
& = - R_k M_{k-1}(b,a) + R_{k+1}(b,a) , \cr
}
$$
since $T_k=T_{k-1}$ in that case, thanks to \eqref{recTk}. 

The second relation for the numerators reads 
$$
(d_k+1)(R_{k+1}(b,a)-R_k(b,a)) = R_{k+1}T_k(b,a) -V_k(b,a).
$$
To that purpose, write
$$
\eqalign{
 b^{t_{k}} a^{\tit_k } ( R_{k+1} (b,a) &  -R_k  (b,a) ) \cr 
& = b^{t_{k}} a^{\tit_k }  R_{k+1} (b,a) + T_k (b,a) - 
\bigl( b^{t_{k}} a^{\tit_k }  R_{k} (b,a) + T_k (b,a) \bigr) \cr
& = R_{k+1} T_k (b,a) -R_k T_k (b,a) = R_{k+1} T_k (b,a) - V_k (b,a). \cr
}
$$

The third relation for the numerators is obvious, while the fourth writes
$$
\eqalign{
(e_k+1)R_{k+1}T_k(b,a) & = R_{k+1}M_k(b,a)-R_{k+1}(b,a))+ (R_{k+1}(b,a)-R_k(b,a))
\cr &= R_{k+1}M_k(b,a)-R_k(b,a),
\cr}
$$
which follows from the equalities
$$
b^{r_k}a^{\tir_k}R_{k+1}T_k(b,a)+R_k(b,a) = R_{k+1}T_kR_k(b,a)= R_{k+1}M_k(b,a),
$$
by Lemma \ref{numerateur} (with $Z=R_k$ and $Y=R_{k+1}T_k$) and Proposition \ref{combsturm}. The fifth relation writes
$$
f_k\bigl((R_{k+1}M_k)(b,a) - R_{k+1}(b,a)\bigr) = V_{k+1}(b,a)- (R_{k+1}T_k)(b,a).
$$
Notice that 
$$
V_{k+1}= R_{k+1}T_{k+1} = R_{k+1}M_k^{b_{k+1}}T_k, 
$$
so that
$$
\eqalign{
V_{k+1}(b,a) - & (R_{k+1}T_k)(b,a) \cr
& = \big(V_{k+1}(b,a) -T_k(b,a)\big)- \big((R_{k+1}T_k)(b,a) -T_k(b,a)\big)
\cr
& = b^{t_k}a^{\tit_k} \bigl( (R_{k+1}M_k^{b_{k+1}})(b,a) -R_{k+1}(b,a) \bigr),
\cr}
$$
thanks to Lemma \ref{numerateur}. Now, we can write  
$$
\eqalign{
(R_{k+1}M_k^{b_{k+1}})(b,a) & -R_{k+1}(b,a) \cr
&= \sum_{j=0}^{b_{k+1}-1}
\bigl( (R_{k+1}M_k^{j+1})(b,a) -(R_{k+1}M_k^j)(b,a) \bigr)
\cr
& = \biggl( \, \sum_{j=0}^{b_{k+1}-1} b^{j q_k}a^{j p_k}\biggr) \bigl( (R_{k+1}M_k)(b,a) -R_{k+1}(b,a)\bigr),
\cr}
$$
by factoring $M_k^j$ on the right and applying again Lemma \ref{numerateur}.  
The fifth relation immediately follows and Lemma \ref{rec(j)k} has been fully checked. 
\end{proof}

We have now all the tools to prove Theorems \ref{fraccont} and \ref{expirrat} . 

\begin{proof}[Proof of Theorem \ref{fraccont}] 
At this stage, the proof of Theorem \ref{fraccont} follows the argumentation of Section 7 in \cite{BuLa21}. We briefly take it again. 

Let us number $\alpha_1, \alpha_2,\dots$ the cyclic sequence
$$
c_0,d_0,1,e_0,f_0,c_1,d_1,1,e_1,f_1,c_2,d_2,1,e_2,f_2,\dots
$$
and define two sequences $(P_j)_{j\ge -1}$ and $(Q_j)_{j\ge -1}$ by the recurrence relations
$$
P_{-1}= b-1, \quad P_0=0, \quad P_j= \alpha_j P_{j-1} + P_{j-2} , \quad j \ge 1,
$$
$$
Q_{-1}= 0, \quad Q_0=b-1, \quad Q_j= \alpha_j Q_{j-1} + Q_{j-2} , \quad j \ge 1.
$$
Equivalently, we have the matrices equalities
$$
\begin{pmatrix}
Q_{j}& Q_{j-1}\cr P_{j} & P_{j-1}
\end{pmatrix}
=
\begin{pmatrix}
b-1& 0\cr 0 & b-1
\end{pmatrix}
\begin{pmatrix}
\alpha_1& 1\cr 1 & 0
\end{pmatrix}
\dots
\begin{pmatrix}
\alpha_j & 1\cr 1 & 0
\end{pmatrix}, 
\quad j\ge 1.
$$
Lemma \ref{rec(j)k} tells us that the sequence $(P_j)_{j\ge 1}$ (resp. $(Q_j)_{j\ge 1}$) coincide with the sequence of numerators (resp. denominators) of 
$$
(1)_0, (2)_0  \ominus (1)_0, (2)_0, (3)_0, (4)_0, (1)_1, (2)_1  \ominus (1)_1, (2)_1, (3)_1, (4)_1,\dots .      
$$
Now, we reduce by associativity the  (formal) infinite product
$$
\begin{pmatrix}
\alpha_1& 1\cr 1 & 0
\end{pmatrix}
\begin{pmatrix}
\alpha_2 & 1\cr 1 & 0
\end{pmatrix}
\begin{pmatrix}
\alpha_3 & 1\cr 1 & 0
\end{pmatrix}
\cdots ,
$$ 
thanks to the processes $(i)$ and $(ii)$ of Theorem \ref{fraccont}. For shortness, write 
$$
E(\alpha)= \begin{pmatrix}
\alpha & 1\cr 1 & 0
\end{pmatrix}.
$$
$(i)$ If $a_{k+2}=b_{k+2}$, we reduce the product of the nine consecutive factors : 
 $$
 \begin{aligned}
 &   \quad E(c_k)E(d_k)E(1)E(e_k)E(f_k) E(c_{k+1})E(d_{k+1})E(1)E(e_{k+1}) = 
 \\
& E(c_k)E(d_k)E(1)E(e_k)E(0) E(-e_k-1)E(d_{k})E(1)E(e_{k+1})
 =  E(c_k+e_{k+1}+1).
 \end{aligned}
 $$
If $c_k= \alpha_l$, we thus have
$$
 \begin{pmatrix}
Q_{l+8}& Q_{l+7}\cr P_{l+8} & P_{l+7}
\end{pmatrix}
=
 \begin{pmatrix}
Q_{l-1}& Q_{l-2}\cr P_{l-1} & P_{l-2}
\end{pmatrix}
\begin{pmatrix}
c_k+e_{k+1}+1 & 1\cr 1 & 0
\end{pmatrix},
$$
and we jump from ${P_{l-1}\over Q_{l-1}}= (4)_{k-1}$ to ${P_{l+8}\over Q_{l+8}}= (3)_{k+1}$, thanks to the elementary matrix $E(c_k+e_{k+1}+1)$. 

\noindent $(ii)$  Reduction $(i)$ enables us to transform  the infinite product $E(\alpha_1)E(\alpha_2)\cdots $ into the product 
$$
E(\alpha_1)E(\alpha_2)\cdots = E(\alpha'_1)E(\alpha'_2)\cdots,
$$
where $\alpha'_1, \alpha'_2, \dots$ are non-negative. We may encounter some zeroes (these precisely occur in the six cases displayed after Theorem \ref{fraccont}). If $\alpha'_l=0$ say,  we replace $\alpha'_{l-1},0,\alpha'_{l+1}$  by $\alpha'_{l-1} + \alpha'_{l+1}$, thanks to the matrices equality
$$
E( \alpha'_{l-1})E(0)E(\alpha'_{l+1})= E(\alpha'_{l-1}+ \alpha'_{l+1}).
$$
 
 We finally end up with a product
 $$
 E(\alpha_1)E(\alpha_2)\cdots = E(A_1)E(A_2)\cdots
$$
 where $A_1,A_2, \dots$ are positive. 
Moreover, every convergent  $[0;A_1,\dots, A_n]$  equals one of the five fractions 
$(1)_k$, $(2)_k \ominus (1)_k, (2)_k, (3)_k, (4)_k$,     
and  infinitely many of these convergents are of the form  $(j)_k$ for some $1\le j \le 4$. Now, by Lemma \ref{numerateur}, $(j)_k= (b-1)\xi_{R_{k+1}\dots}(1/b,1/a)$ for some ultimately periodic word $R_{k+1}\dots$ sharing a large prefix with $\s_{\theta, \rho}$ (or $\s'_{\theta,\rho}$). It follows that
$$
\xi = [0, A_1, A_2, \dots].
$$

Since $a$ and $b$ are integers, observe that the numbers 
$$
c_0, d_0, 1, e_0, f_0, c_1,d_1, 1, e_2, f_2, \dots
$$
 are integers, except possibly 
$c_0 = {b^{a_1-b_1}a-b\over b-1}$. 
But $c_0$ is a non-negative  integer if we assume that $a$ is congruent to $1$ modulo $b-1$. 
Then,  the $A_j $ are positive integers and $A_1, A_2,\dots$ is the sequence of  partial quotients of $\xi$. Theorem \ref{fraccont} is proved.
\end{proof}

\begin{proof}[Proof of Theorem \ref{expirrat}. ] 
Assume first that $a$ is congruent to $1$ modulo $b-1$. 
Let $\xi$ denote one of the numbers $(b-1)\xi_{\s_{\theta,\rho}}(1/b, 1/a)$ or $(b-1)\xi_{\s'_{\theta,\rho}}(1/b, 1/a)$ and let $\xi'$ be the corresponding number with $a=1$. We denote by $(P_j/Q_j)_{j\ge 1}$ (resp. 
$(P'_j/Q'_j)_{j\ge 1}$) the sequence of convergents to $\xi$ (resp. $\xi'$). We claim that 
\begin{equation} \label{Qj}
Q'_j \gg \ll Q_j^{\varphi}, \quad 
\hbox{where} \ \   \varphi = {\log b \over \log b a^\theta}, 
\end{equation}
as $j$ tends to infinity. Indeed, it follows from the proof of Theorem \ref{fraccont} that each convergent $P_j/Q_j$ coincides with  one of the fractions
$$
(1)_k, (2)_k \ominus (1)_k, (2)_k, (3)_k, (4)_k,
$$
for some $k$. Moreover, if $P_j/Q_j= (3)_k$ say, then $P'_j/Q'_j= (3)'_k$,
where $(3)'_k$ stands for the corresponding fraction with $a=1$, observing that the reductions $(i)$ and $(ii)$ occurring in Theorem \ref{fraccont} are independent of $a$ and $b$ (they depend only on the two sequences $(a_k)_{k\ge 1}$ and $(b_k)_{k\ge 1}$). It follows that 
$$
(ba^\theta)^{r_{k+1}+q_k} \ll Q_j = { b^{r_{k+1}}a^{\tir_{k+1}}( b^{q_k}a^{p_k}-1)\over b-1} \ll (ba^\theta)^{r_{k+1}+q_k},
$$
by using Lemma \ref{rktirk}, while
$$
b^{r_{k+1}+q_k} \ll Q'_j = { b^{r_{k+1}}( b^{q_k}-1)\over b-1} \le b^{r_{k+1}+q_k}.
$$
Then, \eqref{Qj} holds true in this case. The other cases are similar.

Now, the theory of continued fractions and \eqref{Qj} yield that
the  irrationality exponents $\mu(\xi)$ and $\mu(\xi')$ of $\xi$ and $\xi'$ are equal, since they are given by the formulae
$$
\mu(\xi) = 1 + \limsup_{j\to +\infty} {\log Q_{j+1}\over \log Q_j} = 
1 + \limsup_{j\to +\infty} {\log Q'_{j+1}\over \log Q'_j} = \mu(\xi').
$$
As already mentioned, Theorem \ref{expirrat} holds true for $\xi'$, and thus for $\xi$,  by \cite[Theorem 2.4]{BuLa21}.

If $a$ is not assumed to be congruent to $1$ modulo $b-1$, then $A_1$ is a positive rational number whose denominator divides $b-1$, while $A_2, A_3, \dots$ are positive integers. Define
$$
{P_j\over Q_j} = [ 0; A_1, \dots, A_j], \quad j\ge 1,
$$
or equivalently
$$
\begin{pmatrix}
Q_{j}& Q_{j-1}\cr P_{j} & P_{j-1}
\end{pmatrix}
=
\begin{pmatrix}
A_1& 1\cr 1 & 0
\end{pmatrix}
\cdots
\begin{pmatrix}
A_j & 1\cr 1 & 0
\end{pmatrix}, 
\quad j\ge 1.
$$
Then, the $P_j$ are integers and the $Q_j$ are rational numbers with denominators dividing $b-1$. The sequence $(P_j/Q_j)_{j\ge 1}$ does not necessarily coincide with the sequence of convergents to $\xi$. However,  the inequalities 
$$
\left\vert \xi - {P_j\over Q_j} \right\vert \le {1\over Q_jQ_{j+1}}, \quad j\ge 1,
$$
remain true, and the above argumentation remains valid.
\end{proof}

\section{Functional equations and expansions of Hecke-Mahler series}
\label{functeq}
 
 We give in this Section  analytical formulae involving Hecke-Mahler series which will reveal 
 to be useful for the proof of Theorem \ref{maintrans}. 
 
%We consider the Sturmian word $\bfs = \bfs_{\theta, \rho}$, but everything below also holds for $\bfs'_{\theta, \rho}$. 

Let us begin with a  relation between the fractions $(3)_k$ and $(4)_k$.
Here and unless otherwise stated, 
we consider $\alpha$ and $\beta$ as variables and work in the ring of power series $\Q [[\alpha, \beta]]$. 
We have

\begin{equation} \label{difference}
(3)_k = (4)_{k-1} + (-1)^k \left({1\over \beta}-1\right)^2{ \beta^{r_{k+1}+q_k}\alpha^{\tir_{k+1}+p_k} \over 1-  \beta^{q_k}\alpha^{p_k}}.
\end{equation}

We stress that here and in this section the $+$ sign between fractions denotes the usual addition and not    
the Farey addition, denoted by $\oplus$ in the previous section.     

The equality \eqref{difference} is proved in \cite[Lemma 2.2]{BKLN}, but only with the case $\alpha=1$. The general case is similar.

For $k\ge 0$,  set $\bfv_k = V_k^\infty$, 
$$
\sigma_k = \sum_{n=1}^{q_k} s_n\beta^n\alpha^{\sum_{h=1}^n s_h}, 
\quad {\rm and} \quad \gamma_k = \beta^{q_k}\alpha^{p_k},
$$
so that 
$$
\xi_{\bfv_k}(\beta,\alpha)= {\sigma_k \over 1-\gamma_k} ={ (4)_{k-1} \over {1\over \beta}-1},
$$
 for $k\ge 1$.  We have
$$
\sigma_0= 0, \quad \sigma_1 = \beta^{a_1-b_1}\alpha , \quad
\gamma_0= \beta, \quad \gamma_1= \beta^{a_1}\alpha.
$$

The recursion relations between the words $V_k$ yield the

\begin{lemma}\label{recsigmak}
For any $k\ge 1$, the numerators $\sigma_k$  satisfy the linear recurrence relation
$$
\sigma_{k+1} = {1-\gamma_{k+1}-\gamma_k^{a_{k+1}-b_{k+1}}(1-\gamma_{k-1})\over 1-\gamma_{k}} \sigma_k + \gamma_k^{a_{k+1}-b_{k+1}} \sigma_{k-1}.
$$
It follows that
$$
\eqalign{
(4)_{k} -(4)_{k-1}  = &(-1)^k {\alpha\beta\left({1\over\beta}-1\right)^2 \prod_{h=0}^k\gamma_h^{a_{h+1}-b_{h+1}}\over (1-\gamma_{k+1})(1-\gamma_k)}
\cr
=& (-1)^k { \alpha\beta\left({1\over\beta}-1\right)^2 \beta^{ \sum_{h=0}^k (a_{h+1}-b_{h+1} )q_h}\alpha^{ \sum_{h=0}^k(a_{h+1}-b_{h+1})p_h}\over (1-\beta^{q_{k+1}}\alpha^{p_{k+1}})(1-\beta^{q_k}\alpha^{p_k})}
\cr
=& (-1)^k \left( {1\over \beta}-1\right)^2 { \beta^{r_{k+1}+q_k}\alpha^{\tir_{k+1}+p_k} \over (1-\beta^{q_{k+1}}\alpha^{p_{k+1}}) (1-  \beta^{q_k}\alpha^{p_k})}.
}
$$
and that 
$$
\begin{aligned}
 &  \quad  \qquad (3)_{k+1}  -(3)_k = 
\\
& 
   (-1)^k \left({1\over \beta}-1\right)^2 
{  \beta^{r_{k+1}+q_{k+1}+ q_{k}}\alpha^{\tir_{k+1}+p_{k+1}+ p_{k}}
 - \beta^{r_{k+2}+q_{k+1}}\alpha^{\tir_{k+2}+p_{k+1}}(1-\beta^{q_k}\alpha^{p_k})  \over  (1- \beta^{q_k}\alpha^{p_k})( 1- \beta^{q_{k+1}}\alpha^{p_{k+1}})}.
\end{aligned}
$$

\end{lemma}

\begin{proof}
Note that $V_k $ has length $q_k$ and contains $p_k$ letters $1$. Then, we deduce from the word equation 
$$
V_{k+1}= V_k^{a_{k+1}-b_{k+1}}V_{k-1}V_k^{b_{k+1}}
$$
the equality
$$
\eqalign{
\sigma_{k+1}= &(1+ \gamma_k+ \cdots  + \gamma_k^{a_{k+1}-b_{k+1}-1})\sigma_k+ \gamma_k^{a_{k+1}-b_{k+1}}\sigma_{k-1}  \cr
& \hskip 10mm 
+ \gamma_k^{a_{k+1}-b_{k+1}}\gamma_{k-1}(1+ \cdots + \gamma_k^{b_{k+1}-1})\sigma_k
\cr
=& { 1- \gamma_k^{a_{k+1}-b_{k+1}}+ \gamma_k^{a_{k+1}-b_{k+1}}\gamma_{k-1} (1-\gamma_k^{b_{k+1}})\over 1-\gamma_k}\sigma_k + 
\gamma_k^{a_{k+1}-b_{k+1}}\sigma_{k-1}
\cr
=& {1-\gamma_{k+1}-\gamma_k^{a_{k+1}-b_{k+1}}(1-\gamma_{k-1})\over 1-\gamma_{k}} \sigma_k + \gamma_k^{a_{k+1}-b_{k+1}} \sigma_{k-1},
}
$$
since $\gamma_k^{a_{k+1}}\gamma_{k-1}= \gamma_{k+1}$. Observe now that the denominators $1-\gamma_k$ satisfy obviously the same linear relation
$$
1-\gamma_{k+1} = {1-\gamma_{k+1}-\gamma_k^{a_{k+1}-b_{k+1}}(1-\gamma_{k-1})\over 1-\gamma_{k}} (1-\gamma_k) + \gamma_k^{a_{k+1}-b_{k+1}}(1- \gamma_{k-1}).
$$
It follows that
$$
(1-\gamma_k)\sigma_{k+1}- (1-\gamma_{k+1})\sigma_k= -\gamma_k^{a_{k+1}-b_{k+1}}( (1-\gamma_{k-1})\sigma_k - (1-\gamma_k)\sigma_{k-1}).
$$
Going down inductively  to $k=1$, we obtain
$$
\eqalign{
(1-\gamma_k)\sigma_{k+1}- & (1-\gamma_{k+1})\sigma_k
=  (-1)^k ((1-\gamma_0)\sigma_1- (1-\gamma_1)\sigma_0) \prod_{h=1}^k\gamma_h^{a_{h+1}-b_{h+1}}  
\cr
= & (-1)^k(1-\beta)\alpha\prod_{h=0}^k\gamma_h^{a_{h+1}-b_{h+1}}
= (-1)^k \left( {1\over \beta}-1\right) \beta^{r_{k+1}+q_k}\alpha^{\tir_{k+1}+p_k}, 
}
$$
by using Lemma \ref{recrk} and noting that  $\sigma_0= 0$ and $\sigma_1 = \beta^{a_1-b_1}\alpha= \gamma_0^{a_1-b_1}\alpha$.
The formulae for $(4)_{k} -(4)_{k-1}$ immediately follow. For  the difference $(3)_{k+1}-(3)_k$, we use moreover  the equality \eqref{difference} to obtain
$$
\begin{aligned}
 & (3)_{k+1}  -(3)_k = 
\\ 
&  (4)_k -(4)_{k-1} + (-1)^k \left({1\over \beta}-1\right)^2\left( - { \beta^{r_{k+2}+q_{k+1}}\alpha^{\tir_{k+2}+p_{k+1}} \over 1-  \beta^{q_{k+1}}\alpha^{p_{k+1}}} - { \beta^{r_{k+1}+q_k}\alpha^{\tir_{k+1}+p_k} \over 1-  \beta^{q_k}\alpha^{p_k}} \right) = 
\\
& 
  (-1)^k \left({1\over \beta}-1\right)^2 \Biggl( { \beta^{r_{k+1} +q_k} \alpha^{\tir_{k+1} +p_{k+1}} \over (1- \beta^{q_k}\alpha^{p_k})( 1- \beta^{q_{k+1}}\alpha^{p_{k+1}}) }
\\
& \quad - {   \beta^{r_{k+2}+q_{k+1}}\alpha^{\tir_{k+2}+p_{k+1}}(1-\beta^{q_k}\alpha^{p_k}) +  \beta^{r_{k+1}+q_{k}}\alpha^{\tir_{k+1}+p_{k}}(1-\beta^{q_{k+1}}\alpha^{p_{k+1}}) \over (1- \beta^{q_k}\alpha^{p_k})( 1- \beta^{q_{k+1}}\alpha^{p_{k+1}})}
\Biggr) =
\\
& 
  (-1)^k \left({1\over \beta}-1\right)^2 
{  \beta^{r_{k+1}+q_{k+1}+ q_{k}}\alpha^{\tir_{k+1}+p_{k+1}+ p_{k}}
 - \beta^{r_{k+2}+q_{k+1}}\alpha^{\tir_{k+2}+p_{k+1}}(1-\beta^{q_k}\alpha^{p_k})  \over  (1- \beta^{q_k}\alpha^{p_k})( 1- \beta^{q_{k+1}}\alpha^{p_{k+1}})}.
\end{aligned}
$$
The proof is complete. 
\end{proof}

\begin{corollary}   \label{reprxi} 
We have the following formulae for the Hecke-Mahler series
$$
\eqalign{
\xi_{\bfs_{\theta, \rho}} (\beta,\alpha) =& (1-\beta)\alpha 
\sum_{k\ge 0}  (-1)^k{\prod_{h=0}^k \gamma_h^{a_{h+1}-b_{h+1}}\over (1-\gamma_{k+1})(1-\gamma_k)}
\cr
= & (1-\beta)\alpha \sum_{k\ge 0}  (-1)^k{\beta^{ \sum_{h=0}^k (a_{h+1}-b_{h+1} )q_h}\alpha^{ \sum_{h=0}^k(a_{h+1}-b_{h+1})p_h}\over (1-\beta^{q_{k+1}}\alpha^{p_{k+1}})(1-\beta^{q_k}\alpha^{p_k})} \cr
= & \frac{1-\beta}{\beta} \sum_{k\ge 0}  (-1)^k{\beta^{r_{k+1} + q_k}\alpha^{ \tir_{k+1} + p_k}\over (1-\beta^{q_{k+1}}\alpha^{p_{k+1}})(1-\beta^{q_k}\alpha^{p_k})}
}
$$
and
$$
\eqalign{
& \xi_{\bfs_{\theta, \rho}} (\beta,\alpha) =   \frac{1-\beta}{\beta} \, \biggl( \, 
{\beta^{r_{1} + 1}\alpha  - \beta^{r_{2} + q_1}\alpha^{ \tir_{2} + p_1} (1 - \beta^{q_0}\alpha^{p_0} )
\over (1-\beta^{q_{1}} \alpha^{p_{1}})(1-\beta^{q_0}\alpha^{p_0})}  \cr 
&  \, \,   +  \sum_{k\ge 1}  (-1)^{k}  { {\beta^{r_{k+1} + q_{k+1} + q_{k}}\alpha^{ \tir_{k+1} + p_{k+1} + p_{k}} 
-  \beta^{r_{k+2} + q_{k+1} }\alpha^{ \tir_{k+2} + p_{k+1}}   
( 1 - \beta^{q_{k}}\alpha^{ p_{k}} ) }
\over (1-\beta^{q_{k+1}}\alpha^{p_{k+1}})( 1-\beta^{q_{k}} \alpha^{p_{k}}) } \biggr). 
\cr} 
$$
\end{corollary}

\begin{proof}

We use the telescopic sums
$$
\eqalign{
\xi_{\bfs_{\theta, \rho}} (\beta,\alpha)& = { (4)_{0}\over {1\over \beta} -1} +\sum_{k\ge 1}  { (4)_k-(4)_{k-1} \over {1\over \beta} -1 } 
\cr
& = {(3)_1  \over {1\over \beta} -1 } 
+\sum_{k\ge 1}  { (3)_{k+1}-(3)_k  \over {1\over \beta} -1 } ,
\cr}
$$
which,  combined with Lemma \ref{recsigmak}, give rise to the terms in the sums with index $k\ge 1$. It remains to compute ${ (4)_{0}/({1\over \beta} -1}) $ and $ (3)_1 /( {1\over \beta} -1 )$.
We have the equalities
$$
\begin{aligned}
 { (4)_{0}\over {1\over \beta} -1} = { \sigma_1\over 1 -\gamma_1} 
 = \frac{\alpha \beta^{a_1 - b_1}}{1 - \alpha \beta^{a_1}} 
& = \alpha (1 - \beta) \frac{\gamma_0^{a_1 - b_1}}{(1 - \gamma_0) (1 - \gamma_1)} 
\\
& ={1-\beta \over \beta} {\beta^{r_1+q_0}\alpha^{\tir_1+p_0} \over (1- \beta^{q_1}\alpha^{p_1})(1- \beta^{q_0}\alpha^{p_0})},
\end{aligned}
$$
which establish the three first expressions for $\xi_{\bfs_{\theta, \rho}} (\beta,\alpha)$.

We now deal with 
$$
{(3)_1  \over {1\over \beta} -1 } = \xi_{R_{2}M_1^{\infty}}(\beta,\alpha).
$$
Assume first that $a_2-b_2 \ge 1$. Then,
$$
R_2 M_1^\infty=  R_1M_1^{a_2-b_2-1}M_0 M_1^\infty = 0^{a_1-b_1-1}1(0^{a_1-1}1)^{a_2-b_2-1}0(0^{a_1-1}1)^\infty.
$$
It follows that
$$
\begin{aligned}
\xi_{R_{2}M_1^{\infty}}(\beta,\alpha)=& \beta^{a_1-b_1}  \alpha\Bigl( 1 + \beta^{a_1}\alpha + \cdots + (\beta^{a_1}\alpha)^{a_2-b_2-1}\Bigr)
\\
& \qquad  \quad + \beta^{a_1-b_1+a_1(a_2-b_2-1)+1}\alpha^{a_2-b_2}{\beta^{a_1}\alpha\over 1- \beta^{a_1}\alpha}
\\
 =& { \beta^{u_1}\alpha^{v_1} -\beta^{u_2}\alpha^{v_2}+\beta^{u_3}\alpha^{v_3}\over 1-\beta^{a_1}\alpha},
\end{aligned}
$$
with
$$
\begin{aligned}
u_1 &= a_1-b_1 = r_1,
 \quad v_1 = 1 ,
 \\
u_2 &= a_1-b_1 +a_1(a_2-b_2) =r_2 +q_1-1,
 \quad v_2 = a_2-b_2 +1 = \tir_2 +p_1 ,
 \\
 u_3 &= a_1a_2+1 +a_1 -b_1 -b_2a_1 =r_2 +q_1,
 \quad v_2 = a_2-b_2 +1 = \tir_2 +p_1 .
\end{aligned}
$$
Thus,
$$
{(3)_1  \over {1\over \beta} -1 }=
{1-\beta \over \beta} { \beta^{r_1+1}\alpha - \beta^{r_2+q_1}\alpha^{\tir_2+p_1}(1-\beta)\over (1-\beta^{a_1}\alpha)(1-\beta)},
$$
as asserted. The fourth expression for $\xi_{\bfs_{\theta, \rho}} (\beta,\alpha)$ is established when $a_2-b_2 \ge 1$. In the case $a_2=b_2$, we have $R_2=0$, $r_1=a_1$. The computations are similar and simpler. 

\end{proof}

As an example, for the characteristic Sturmian word $\bc_\theta$
we have $b_k = 0$ for every $k\ge 1$. Then, it follows from \eqref{sommeqk} and \eqref{sommepk} that
$$
 \sum_{h=0}^k (a_{h+1}-b_{h+1} )q_h= q_{k+1}+ q_k-1  
\quad
 {\rm and}
 \quad
 \sum_{h=0}^k (a_{h+1}-b_{h+1} )p_h= p_{k+1}+ p_k-1.    
$$
Thus, we recover the known formula
$$
\xi_{\bc_\theta}(\beta,\alpha)= ({1\over \beta}-1)\sum_{k\ge 0} (-1)^k {  \beta^{ q_{k+1}+q_k}\alpha^{ p_{k+1}+ p_k}\over (1-\beta^{q_{k+1}}\alpha^{p_{k+1}})(1-\beta^{q_k}\alpha^{p_k})},
$$
which is usually obtained as a consequence of the functional equation for the Hecke-Mahler series. 

Conversely,  a functional chain of equations  of Mahler's type can be deduced from 
our  formula for an arbitrary Sturmian word $\bfs$. 
%For $m\ge 0$, put $\theta_m = [0, a_{m+1}, a_{m+2}, \dots]$ and denote by $\bfs_m$ 
%the Sturmian word with slope $\theta_m$ and formal intercept $b_{m+1}, b_{m+2}, \dots$. 
For $m\ge 0$, put 
$\theta_m = [0, a_{m+1}, a_{m+2}, \dots]$ and denote by $\bfs_m$ 
the Sturmian word with slope $\theta_m$ and formal intercept $b_{m+1}, b_{m+2}, \dots$ (see Definition \ref{formalint}). 
Observe that $\bfs = \bfs_0$. 
With our notation, we have 
$\xi_{\bfs}(\beta,\alpha) =  \xi_{\bfs_0} (\gamma_0, \gamma_{-1})$, 
where $\gamma_{-1} = \beta^{q_{-1}} \alpha^{p_{-1}} = \alpha$. 

\begin{proposition}   \label{xisxism} 
For any $m\ge 1$, we have the relation of Mahler's type
$$
\eqalign
{
\xi_{\bfs}(\beta,\alpha)
= & (1-\beta)\alpha\sum_{k= 0}^{m-1}  (-1)^k{\beta^{ \sum_{h=0}^k (a_{h+1}-b_{h+1} )q_h}\alpha^{ \sum_{h=0}^k(a_{h+1}-b_{h+1})p_h}\over (1-\beta^{q_{k+1}}\alpha^{p_{k+1}})(1-\beta^{q_k}\alpha^{p_k})}  
\cr
& +  (-1)^m{(1-\beta)\alpha\over (1-\beta^{q_m}\alpha^{p_m})\beta^{q_{m-1}}\alpha^{p_{m-1}}}\beta^{ \sum_{h=0}^{m-1} (a_{h+1}-b_{h+1} )q_h}\alpha^{ \sum_{h=0}^{m-1}(a_{h+1}-b_{h+1})p_h}
 \cr
& \hskip 62mm \times
\xi_{\bfs_m}(\beta^{q_m}\alpha^{p_m}, \beta^{q_{m-1}}\alpha^{p_{m-1}})
\qquad
\cr
=(1-\beta)\alpha & \biggl( \sum_{k= 0}^{m-1}  (-1)^k{\prod_{h=0}^k\gamma_h^{a_{h+1}-b_{h+1}}\over (1-\gamma_{k+1})(1-\gamma_k)} 
+ (-1)^m { \prod_{h=0}^{m-1}\gamma_h^{a_{h+1}-b_{h+1}}\over (1-\gamma_m)\gamma_{m-1}} \xi_{\bfs_m}(\gamma_m,\gamma_{m-1})
\biggr)
\cr
& = (1-\beta)\alpha \biggl( {\sigma_m \over 1 - \gamma_m} 
+ (-1)^m { \prod_{h=0}^{m-1}\gamma_h^{a_{h+1}-b_{h+1}}\over (1-\gamma_m)\gamma_{m-1}} \xi_{\bfs_m}(\gamma_m,\gamma_{m-1}) \biggr). 
\cr
}
$$
\end{proposition}

\begin{proof}
We truncate the sum giving $\xi_\bfs(\beta,\alpha)$ at the order $m$ and consider the remaining terms
$$
\eqalign{
& (1-\beta)\alpha  \sum_{k\ge m} (-1)^k {  \beta^{ \sum_{h=0}^k (a_{h+1}-b_{h+1} )q_h}\alpha^{ \sum_{h=0}^k(a_{h+1}-b_{h+1})p_h}\over (1-\beta^{q_{k+1}}\alpha^{p_{k+1}})(1-\beta^{q_k}\alpha^{p_k})}
\cr
& =
(-1)^m(1-\beta)\alpha\beta^{ \sum_{h=0}^{m-1} (a_{h+1}-b_{h+1} )q_h}\alpha^{ \sum_{h=0}^{m-1}(a_{h+1}-b_{h+1})p_h}
\cr
& \hskip 5mm \times 
\sum_{k\ge 0}  (-1)^k{\beta^{ \sum_{h=0}^k (a_{m+h+1}-b_{m+h+1} )q_{m+h}}\alpha^{ \sum_{h=0}^k(a_{m+h+1}-b_{m+h+1})p_{m+h}}\over (1-\beta^{q_{m+k+1}}\alpha^{p_{m+k+1}})(1-\beta^{q_{m+k}}\alpha^{p_{m+k}})}. 
}
$$
Now, we claim that the last factor
$$
\sum_{k\ge 0}  (-1)^k{\beta^{ \sum_{h=0}^k (a_{m+h+1}-b_{m+h+1} )q_{m+h}}\alpha^{ \sum_{h=0}^k(a_{m+h+1}-b_{m+h+1})p_{m+h}}\over (1-\beta^{q_{m+k+1}}\alpha^{p_{m+k+1}})(1-\beta^{q_{m+k}}\alpha^{p_{m+k}})}
$$
is equal to 
$$
(1-\beta^{q_m}\alpha^{p_m})^{-1}( \beta^{q_{m-1}}\alpha^{p_{m-1}})^{-1}\xi_{\bfs_m}(\beta^{q_m}\alpha^{p_m}, \beta^{q_{m-1}}\alpha^{p_{m-1}}). 
$$
Indeed, let $(u_n/v_n)_{n\ge 0}$ be the convergents of $\theta_m$. We have 
$$
{u_0\over v_0}= { 0\over1},\quad {u_1\over v_1}= {1\over a_{m+1}}, \quad{u_2\over v_2}= {a_{m+2}\over a_{m+1}a_{m+2}+1}, \dots , 
$$
and we easily check that, for any $h\ge 0$, we have
$$
q_{m+h} = v_h q_m + u_h q_{m-1}  \quad{\rm and} \quad p_{m+h} = v_h p_m + u_h p_{m-1}.
$$
It follows that we can write the exponents in a form  involving the convergents of $\theta_m$ :
$$
\eqalign{
\sum_{h=0}^k (a_{m+h+1} & -b_{m+h+1} )q_{m+h} \cr
& = \Big(\sum_{h=0}^k (a_{m+h+1}-b_{m+h+1} )v_{h}\Big) q_m+ \Big(\sum_{h=0}^k (a_{m+h+1}-b_{m+h+1} )u_{h}\Big )q_{m-1} \cr
}
$$
and
$$
\eqalign{
\sum_{h=0}^k (a_{m+h+1} & -b_{m+h+1} )p_{m+h} \cr
& = \Big(\sum_{h=0}^k (a_{m+h+1}-b_{m+h+1} )v_{h}\Big) p_m+ \Big(\sum_{h=0}^k (a_{m+h+1}-b_{m+h+1} )u_{h}\Big )p_{m-1}. \cr
}
$$
Thus
$$
\displaylines
{
\beta^{ \sum_{h=0}^k (a_{m+h+1}-b_{m+h+1} )q_{m+h}}\alpha^{ \sum_{h=0}^k(a_{m+h+1}-b_{m+h+1})p_{m+h}} =
\cr
(\beta^{q_m}\alpha^{p_m})^{\sum_{h=0}^k (a_{m+h+1}-b_{m+h+1} )v_{h}}\times (\beta^{q_{m-1}}\alpha^{p_{m-1}})^{\sum_{h=0}^k (a_{m+h+1}-b_{m+h+1} )u_{h}},
}
$$
and 
$$
\displaylines{
1-\beta^{q_{m+k+1}}\alpha^{p_{m+k+1}}= 1-(\beta^{q_m}\alpha^{p_m})^{v_{k+1}}(\beta^{q_{m-1}}\alpha^{p_{m-1}})^{u_{k+1}}
\cr
1-\beta^{q_{m+k}}\alpha^{p_{m+k}}= 1-(\beta^{q_m}\alpha^{p_m})^{v_{k}}(\beta^{q_{m-1}}\alpha^{p_{m-1}})^{u_{k}}.
}
$$
The last claim follows from the equalities
$$
(1-\beta)\alpha   \biggl( \sum_{k= 0}^{m-1}  (-1)^k{\prod_{h=0}^k\gamma_h^{a_{h+1}-b_{h+1}}
\over (1-\gamma_{k+1})(1-\gamma_k)} \biggr) 
= {\sigma_m \over 1 - \gamma_m} - {\sigma_0 \over 1 - \gamma_0}
= {\sigma_m \over 1 - \gamma_m}. 
$$
The proof is complete. 
\end{proof}

\section{When the slope has unbounded partial quotients}    \label{unbounded}

The purpose of this Section is to establish Theorem \ref{maintrans} when the slope $\theta$ has unbounded partial quotients. 
In this case, an application of Liouville's inequality is sufficient to conclude. 
We use the logarithmic Weil height $h$ and Liouville's inequality under the form
\begin{equation} \label{liouv}
\log | \zeta | \ge - [\Q(\zeta) : \Q] \, h(\zeta),
\end{equation} 
for any nonzero algebraic number $\zeta$. 
There is some similarity with the proof of \cite[Theorem 6]{Ko96b}. 

We make use of the approximations ${\beta\over 1-\beta}(4)_{k-1}$ and 
${\beta\over 1-\beta}(3)_{k}$ to $\xi = \xi_{\bfs_{\theta, \rho}} (\beta,\alpha)$, considered in Section \ref{functeq}. If $U$ and $V$ are positive quantities depending upon $k$, let us write  $U\asymp V$ to indicate   
that there exist positive constants $c, c'$ such that the inequalities $cU \le V \le c'U$ hold for large $k$. 

\begin{lemma}\label{3k4k}
We have the two estimates
$$
\left\vert \xi - {\beta\over 1-\beta}(4)_{k-1}\right\vert \asymp (\vert\beta \alpha^\theta \vert)^{u_k+q_k} \quad \text{  with} \quad u_k = 
\begin{cases} 
r_{k+1} & \text {if } \quad a_{k+2} -b_{k+2}\ge 1,
\\
r_k +q_{k+1} & \text{if} \quad a_{k+2}=b_{k+2}, 
\end{cases}
$$
and 
$$
  \left\vert \xi -  {\beta\over 1-\beta}(3)_{k}\right\vert \asymp  (\vert\beta \alpha^\theta \vert)^{v_k+q_{k+1}} 
$$
  with
  $$
 v_k = 
\begin{cases} 
r_{k+1} +q_k & \text {if } \quad a_{k+2} -b_{k+2}\ge 2,
\\
r_{k+1} +2q_{k} & \text{if} \quad a_{k+2}-b_{k+2}=1, a_{k+3}-b_{k+3}\ge 1,
\\
r_{k+1} +q_{k} & \text{if} \quad a_{k+2}=1, b_{k+2}=0,  a_{k+3}=b_{k+3},
\\
r_{k} & \text{if} \quad a_{k+2} = b_{k+2}.
\end{cases}
$$
\end{lemma}

\begin{proof}
Let us set,  for $k\ge 1$,
$$
\Gamma_k =\frac{1-\beta}{\beta}   (-1)^k{\beta^{r_{k+1} + q_k}\alpha^{ \tir_{k+1} + p_k}\over (1-\beta^{q_{k+1}}\alpha^{p_{k+1}})(1-\beta^{q_k}\alpha^{p_k})}, 
$$
and
$$
 \Delta_k  
=  \frac{1-\beta}{\beta}  (-1)^{k}  { {\beta^{r_{k+1} + q_{k+1} + q_{k}}\alpha^{ \tir_{k+1} + p_{k+1} + p_{k}} 
-  \beta^{r_{k+2} + q_{k+1} }\alpha^{ \tir_{k+2} + p_{k+1}}   
( 1 - \beta^{q_{k}}\alpha^{ p_{k}} ) }
\over (1-\beta^{q_{k+1}}\alpha^{p_{k+1}})( 1-\beta^{q_{k}} \alpha^{p_{k}}) }. 
$$
Recalling Lemma \ref{recsigmak} and the representations of $\xi = \xi_{\bfs_{\theta, \rho}} (\beta,\alpha)$ given in Corollary \ref{reprxi}, 
 we have
$$
 \xi - {\beta\over 1-\beta}(4)_{k-1}= \sum_{h\ge k} \, \Gamma_h ,
\quad {\rm and } \quad 
  \xi - {\beta\over 1-\beta}(3)_{k}=  \sum_{h \ge k} \, \Delta_h.
$$
We now estimate the two above sums. For the sum $ \sum_{h\ge k} \, \Gamma_h$, observe that the two sequences of exponents
$$
r_{h+1}+q_h = 1 + \sum_{j= 0}^h (a_{j+1}-b_{j+1})q_j, \quad h = k, k+1, \dots
$$
and 
$$
 \tir_{h+1}+p_h = 1 + \sum_{j= 0}^h (a_{j+1}-b_{j+1})p_j , \quad  h = k, k+1, \dots
$$
occurring in the quantities 
$\Gamma_h$, are non-decreasing. Moreover, $r_{h+1}+q_h = r_{h+2}+q_{h+1}$ if and only if $a_{h+2}=b_{h+2}$, and 
 $r_{h+2} +q_{h+1}\ge r_{h+1} +q_h + q_{h+1}$ if $a_{h+2}>b_{h+2}$.
 Notice also that we cannot have $r_{h+1}+q_h = r_{h+2}+q_{h+1}=r_{h+3}+q_{h+2}$, since the simultaneous equalities $a_{h+2}=b_{h+2}$ and $a_{h+3}=b_{h+3}$ are forbidden according to  Ostrowski's rules. 
 
In view of Lemma \ref{rktirk}, we have
$$
 |\Gamma_h| \asymp (\vert\beta \alpha^\theta \vert)^{r_{h+1} + q_h}.
$$
In order to estimate $ \sum_{h\ge k} \, \Gamma_h$, we distinguish two cases. Assume first that $a_{k+2}-b_{k+2}\ge 1$. Then 
$$
| \Gamma_k | \asymp (\vert\beta \alpha^\theta \vert)^{r_{k+1}+q_k} \quad { \rm and} \quad 
| \Gamma_{h} | \ll (\vert\beta \alpha^\theta \vert)^{r_{k+1}+q_k+ q_{k+1} }, \quad h \ge k+1.
$$
Taking into account the preceding observations, it follows that 
$$
|  \sum_{h\ge k} \, \Gamma_h | \asymp (\vert\beta \alpha^\theta \vert)^{r_{k+1}+q_k}.
$$
Assume secondly that $a_{k+2}=b_{k+2}$. Then,
$$
\eqalign{
|\Gamma_k + \Gamma_{k+1}| & = \biggl| \frac{1 - \beta}{\beta} \biggr| \, 
\biggl| {\beta^{r_{k+1} + q_k}\alpha^{ \tir_{k+1} + p_k} \over 1-\beta^{q_{k+1}}\alpha^{p_{k+1}} } 
\biggl( \frac{1}{1 -  \beta^{q_{k}}\alpha^{p_{k}}} 
- \frac{1}{1 - \beta^{q_{k+2}}\alpha^{p_{k+2}}} \biggr) \biggr| \cr
& = \biggl| \frac{1 - \beta}{\beta} \cdot 
{\beta^{r_{k+1} + q_k}\alpha^{ \tir_{k+1} + p_k} 
(\beta^{q_{k+2}}\alpha^{p_{k+2}} - \beta^{q_{k}}\alpha^{p_{k}} )
\over (1-\beta^{q_{k+1}}\alpha^{p_{k+1}} ) (1 - \beta^{q_{k}}\alpha^{p_{k}}) (1 - \beta^{q_{k+2}}\alpha^{p_{k+2}}) } \biggr|, \cr
}
$$
so that 
$$
|\Gamma_k + \Gamma_{k+1}|  \asymp  (\vert\beta \alpha^\theta \vert)^{r_{k+1} + 2 q_k} = (\vert\beta \alpha^\theta \vert)^{r_{k} +  q_k+q_{k+1}}.
$$
Now, since $a_{k+3}> b_{k+3}$, we get 
$$
r_{k+3} + q_{k+2} - (r_{k} + q_{k+1}) =r_{k+3} + q_{k+2} - (r_{k+2} + q_{k+1}) \ge q_{k+2},
$$
so that
$$
|\sum_{h\ge k+2}  \Gamma_{h} | \ll (\vert\beta \alpha^\theta \vert)^{r_{k}+q_{k+1}+q_{k+2} }. 
$$
It follows that 
$$
|  \sum_{h\ge k} \, \Gamma_h | \asymp (\vert\beta \alpha^\theta \vert)^{r_{k}+q_k+q_{k+1}}.
$$

We now briefly deal with the sum $\sum_{h\ge k} \Delta_h$. Observe that
$$
| \Delta_h | \asymp 
\begin{cases}
 |\beta \alpha^\theta|^{r_{h+1} + q_{h+1} + q_h} & \text{if} \quad a_{h+2}-b_{h+2}\ge 2,
 \\
 |\beta \alpha^\theta|^{r_{h+1} + q_{h+1} + 2 q_h} & \text{if}\quad  a_{h+2}-b_{h+2}= 1,
 \\
  |\beta \alpha^\theta|^{r_{h} + q_{h+1} } & \text{if} \quad a_{h+2}=b_{h+2}.
  \end{cases}
  $$
Looking at  the absolute value of $\Delta_k$ and $\Delta_{k+1}$ according to the above cases,  we check that 
$$
\vert \sum_{h \ge k}  \Delta_k\vert \asymp \vert \Delta_k \vert, 
$$
unless $a_{k+2} =1, b_{k+2}=0$ and  $a_{k+3}=b_{k+3}$, in which case
$$
\vert \sum_{h \ge k}  \Delta_k\vert \asymp \vert \Delta_{k+1} \vert . 
$$
It follows that 
$$
\vert \sum_{h \ge k}  \Delta_k\vert \asymp | \beta\alpha^\theta |^{v_k+q_{k+1}},
$$
as asserted. Lemma \ref{3k4k} is proved.  
\end{proof}

We are now able to prove Theorem \ref{maintrans} when $\theta$ has unbounded partial quotients. 
Assume on the contrary that $\xi $ is algebraic. We  distinguish two cases.

Assume first that $r_{k+1}/q_k$ takes arbitrarily large values and set
$$
\zeta = \xi - {\beta\over 1-\beta} (4)_{k-1}.
$$
Lemma \ref{3k4k} yields,  for large $k$,  that $\zeta$ is non-zero and that
$$
\log | \zeta | \ll -(u_{k} +q_k)  \ll -r_{k+1}, 
$$
since we have always $u_k \ge r_{k+1}$. But  the algebraic number $\zeta $ has height 
$
h(\zeta) \ll q_k.
$
This contradicts Liouville's inequality \eqref{liouv},   
provided that we have chosen $k$ such that $r_{k+1}/q_k$ is large enough.

Assume now that the sequence $(r_{k+1}/q_k)_{k\ge 1}$ is bounded. Set now 
$$
\zeta = \xi - {\beta\over 1-\beta} (3)_{k}.
$$
Again Lemma \ref{3k4k} implies   that $\zeta$ is non-zero and that
$$
\log | \zeta | \ll -( v_k+q_{k+1} ) \ll -q_{k+1}, 
$$
when $k$ is large enough. But  the algebraic number $\zeta $ has now height 
$$
h(\zeta) \ll r_{k+1}+q_k \ll q_k,
$$
by assumption. We get a final contradiction with Liouville's inequality \ref{liouv}, provided that we have chosen $k$ such that $q_{k+1}/q_k$ is large enough.

\section{Functional transcendence}

A general idea underlying Mahler's method is that the transcendence of a function $f(z)$ over $\Q(z)$ 
is transferred to the transcendence of the value of $f$ at every nonzero algebraic point in the open unit disc. 
Therefore, we need a functional transcendence statement. 

\begin{proposition}   \label{hecke}
Let $\theta, \rho$ be real numbers such that $0 \le \theta, \rho < 1$ and $\theta$ irrational.
Then, the function 
$z \mapsto \xi_{\s_{\theta,\rho}} (z, 1)$ is transcendental over $\C(z)$. Consequently,  the function 
$(z_1, z_2) \mapsto \xi_{\s_{\theta,\rho}} (z_1, z_2)$ is transcendental over $\C (z_1, z_2)$. 
\end{proposition}

\begin{proof}
Observe that an algebraic  function, say $f(z)$, holomorphic in the open unit disc,    
can be analytically prolongated in a neighborhood of a point $z_0$ on the unit circle,     
if we assume  that $z_0$ is not a root of the discriminant of the minimal polynomial of $f(z)$ over $\C(z)$.

Therefore, it is sufficient to show that 
$z \mapsto \xi_{\s_{\theta,\rho}} (z, 1)$ cannot be prolongated beyond the unit circle. 
The case $\rho = 0$ has been treated by Hecke \cite{He22}. His argument extends easily to 
an arbitrary value of $\rho$. For the sake of completeness, we give the details below. 
Set
$$
F(z) = \sum_{n \ge 1} \, \{n \theta + \rho\} z^n.
$$
Recall that if, for a power series $\sum_{n \ge 1} c_n z^n$, we have
$$
\lim_{t \to + \infty} \, {1 \over t} \, \sum_{n=1}^t c_n = c,
$$
then
$$
\lim_{r \to 1_-} (1 - r) \sum_{n=1}^{+ \infty} c_n r^n = c,
$$
where $r \to 1_-$ means that the real number $r$ tends to $1$ and is less than $1$. 
Let $t$ be a positive integer.  
Write
$$
S(t) = \sum_{n=1}^t   \{n \theta + \rho \} \rme^{2 \rmi \pi n \alpha} 
$$
and take $\alpha = q \theta + p$, for integers $p, q$ with $q$ nonzero. We have
$$
S(t) = \sum_{n=1}^t   \{n \theta + \rho \} \rme^{2 \rmi \pi  n (q \theta + p)}
=  \rme^{- 2 \rmi \pi q \rho} \, \sum_{n=1}^t   \{n \theta + \rho \} \rme^{2 \rmi \pi q (n \theta + \rho)}. 
$$
As $\theta$ is irrational, the sequence $(\{n \theta + \rho \})_{n \ge 1}$ is 
equidistributed in $[0, 1]$, thus 
$$
\lim_{t \to + \infty} \, {1 \over t} \, \sum_{n=1}^t f(\{n \theta + \rho \}) 
= \int_0^1 f(x) {\rm d}x,
$$
for every continuous function $f$. Consequently, 
$$
\lim_{t \to + \infty} \, {1 \over t} \, S(t) = { \rme^{- 2 \rmi \pi q \rho} \over  2 \rmi \pi q }. 
$$
It then follows that 
$$
\lim_{r \to 1_-} (1 - r) \sum_{n=1}^{+ \infty}   \{n \theta + \rho \} (r \rme^{2 \rmi \pi  (q \theta + p)})^n = 
\lim_{r \to 1_-} (1 - r)  F (r \rme^{2 \rmi \pi  (q \theta + p)})  = 
 { \rme^{- 2 \rmi \pi q \rho} \over  2 \rmi \pi q }. 
$$
Since the set of points of the form $q \theta + p$ is dense modulo one, the function $F$ 
cannot be prolongated beyond the unit circle. The same conclusion holds for 
the function $z \mapsto \sum_{n \ge 1} \, \lfloor n \theta + \rho\rfloor z^n$. 
\end{proof}

\section{Transcendence of Hecke--Mahler series at algebraic points}

Loxton and van der Poorten \cite{LovdP77c} (see also \cite[Section 2.9]{Nish96}) 
obtained a general transcendence theorem for chains of functional equations of Mahler's type, from which 
they deduced \cite[Theorem 8]{LovdP77c} the transcendence of $F_{\theta, 0} (\beta, \alpha)$, for every irrational 
number $\theta$ in $(0, 1)$ and every nonzero complex algebraic numbers $\alpha, \beta$ with 
$|\beta \alpha^\theta| < 1$ and $\beta^{q_k} \alpha^{p_k} \not= 1$ for $k \ge 1$, where $p_k / q_k$ is the $k$-th convergent 
to $\theta$. 

We follow the presentation of Nishioka \cite{Nish96}, with some simplification and modernization. 
In her book, the size $\| \alpha \|$ of an algebraic number $\alpha$ is the maximum of the absolute
values of the conjugates of $\alpha$ and of its denominator. The function $\log \| \cdot \|$ is thus 
comparable to the logarithmic Weil height $h$, which we are using. 

For a $2 \times 2$ matrix $\Omega = (\omega_{i,j})$ with nonnegative integer coefficients and a 
point $(z_1, z_2)$ in $\C^2$, we define an application $\Omega : \C^2 \to \C^2$ by
$$
\Omega (z_1, z_2) = (z_1^{\omega_{1,1}} z_2^{\omega_{1,2}}, z_1^{\omega_{2,1}} z_2^{\omega_{2,2}}).
$$

Let $(\Omega_k)_{k \ge 1}$ be a sequence of matrices with nonnegative integer coefficients. 
Let $K$ be a number field and $\alpha_1, \alpha_2$ nonzero elements in $K$. 
Write 
$$
(\alpha_1^{(k)}, \alpha_2^{(k)})  = \Omega_k (\alpha_1, \alpha_2), \quad k \ge 1.
$$  
Let $f_k(z_1, z_2)$, $k \ge 0$, be in $\Z[[z_1, z_2]]$ with bounded coefficients. 
Write 
$$
f_k (z_1, z_2) = \sum_{\lambda_1, \lambda_2 \ge 0} \, \sigma_{\lambda_1, \lambda_2}^{(k)} z_1^{\lambda_1} z_2^{\lambda_2}, 
\quad k \ge 0,
$$
and $\usigma^{(k)} = ( \sigma_{\lambda_1, \lambda_2}^{(k)} )_{\lambda_1, \lambda_2 \ge 0}$. 
For a collection $\us = (s_{\lambda_1, \lambda_2})_{\lambda_1, \lambda_2 \ge 0}$ of variables, set
$$
F(z_1, z_2; \us) = \sum_{\lambda_1, \lambda_2 \ge 0} \, s_{\lambda_1, \lambda_2}  z_1^{\lambda_1} z_2^{\lambda_2}.
$$
Then, we have
$$
F(z_1, z_2; \usigma^{(k)}) = f_k (z_1, z_2), \quad k \ge 0.
$$

Assume that there exist positive real numbers $r_1, r_2, \ldots$  % and absolute real numbers $c_1, c_2, \ldots$ 
such that $(r_k)_{k \ge 1}$ tends to infinity and 

\smallskip

(i) Every coefficient of $\Omega_k$ is $\ll  r_k$, for $k \ge 1$. 

(ii) There exist positive real numbers $\eta_1, \eta_2$ which are linearly independent over the rationals and such that 
$$
\log |\alpha_i^{(k)}| \sim - \eta_i r_k, \quad i=1, 2, \quad \hbox{as $k$ tends to infinity}.
$$

(iii) For $k \ge 1$, there exist $a_k, b_k$ in $K$ such that 
$$
f_k (\Omega_k (\alpha_1, \alpha_2)) = a_k f_0 (\alpha_1, \alpha_2) + b_k
$$
and
$$
h(a_k), h(b_k) \ll r_k. 
$$

(iv) If $p$ is a positive integer, $P_0 (z_1, z_2;  \us), \ldots , P_p (z_1, z_2; \us)$ are 
polynomials in $z_1, z_2$ and in the variables $s_{\lambda_1, \lambda_2}$, with coefficients in $K$, and 
$$
E(z_1, z_2; \us) = \sum_{j=0}^p P_j (z_1, z_2; \us) F(z_1, z_2; \us)^j = 
\sum_{\lambda_1, \lambda_2 \ge 0} P_{\lambda_1, \lambda_2} (\us) z_1^{\lambda_1} z_2^{\lambda_2}, 
$$
then there exist nonnegative $\lambda_1, \lambda_2$  with the following property: 
%exists a positive integer $I$ with the following property: 
If $k$ is sufficiently large and $P_0 (z_1, z_2;  \usigma^{(k)}), \ldots , P_p (z_1, z_2; \usigma^{(k)})$ are not all zero, 
then 
%there exist $\lambda_1, \lambda_2$ such that $\lambda_1 + \lambda_2 \le I$ and 
$P_{\lambda_1, \lambda_2} (\usigma^{(k)})$ is 
non-zero. 

Assumptions (i), (ii), and (iii) correspond exactly to Assumptions (I), (II), and (III) in \cite{Nish96}. 
Our assumption (iv) is a simplified version of Assumption (V) in \cite{Nish96}. Assumption (IV) in \cite{Nish96} is 
clearly satisfied since the coefficients of the series $f_k$ are integers and are bounded.

\begin{theorem}[Loxton--van der Poorten]     \label{LvdPchaines}
Under the above assumption, the complex number $f_0 (\alpha_1, \alpha_2)$ is transcendental. 
\end{theorem}

Our presentation slightly differs from that of \cite{LovdP77c}, where the 
authors have to cope with admissibility conditions on $\alpha_1$ 
and $\alpha_2$. Here, we have expressed Assumption (iii) with $a_k, b_k$ in $K$, and not 
with functions $a_k (\alpha_1, \alpha_2), b_k (\alpha_1, \alpha_2)$ in $K (\alpha_1, \alpha_2)$, in which case we 
should have excluded the pairs $(\alpha_1, \alpha_2)$ at which these functions are not defined. 
To overcome this difficulty, Nishioka \cite[p. 77]{Nish96} assumes that $\alpha_1$ and $\alpha_2$ are in the open unit disc, but
this is quite restrictive. 

We show how Theorem \ref{LvdPchaines} applies to establish Theorem \ref{maintrans} when 
the slope $\theta$ has bounded partial quotients. 

For $m \ge 0$, recall that $\theta_m = [0, a_{m+1}, a_{m+2}, \dots]$ and that $\bfs_m$ denotes
the Sturmian word with slope $\theta_m$ and formal intercept $b_{m+1}, b_{m+2}, \dots$ (see Definition \ref{formalint}), as in 
Section \ref{functeq}. 
%for $m \ge 0$  the notations 
%$\theta_m = [0, a_{m+1}, a_{m+2}, \dots]$ and $\bfs_m$ 
%introduced in 
Let us start by the following estimates.

\begin{proposition}  \label{chaine}
Let $\alpha$ and $\beta$ be complex numbers such that $0 < |\beta \alpha^\theta |< 1$ and $\beta \not=1$. 
If  there exists $\ell$ such that $\beta^{q_\ell} \alpha^{p_\ell} = 1$, then put $m_0 = \ell + 1$, otherwise put $\ell = -1$ and $m_0 = 1$. 
For any $m \ge m_0$, there exist $A_m$ and $B_m$ such that 
$$
%\xi_{\bfs_m}(\gamma_m,\gamma_{m-1}) = A_m \xi_{\bfs}(\beta,\alpha) + B_m
\xi_{\bfs_m}(\beta^{q_m}\alpha^{p_m}, \beta^{q_{m-1}}\alpha^{p_{m-1}}) = A_m \xi_{\bfs}(\beta,\alpha) + B_m
$$
and
$$
h(A_m), h(B_m) \ll_{\alpha, \beta} q_m.    %\bigl (h(\alpha) + h(\beta) \bigr) q_m.
$$
\end{proposition} 

\begin{proof}[Proof of Proposition \ref{chaine}] 
Here, $\alpha$ and $\beta$ denote complex numbers satisfying $0 < |\beta \alpha^\theta| < 1$ and $\beta \not=1$. 
Recall that $\gamma_k = \beta^{q_k} \alpha^{p_k}$, for $k \ge 0$. 
%We have to take care at possible vanishing denominators. 
Let $m \ge m_0$ be an integer. Then, $\gamma_m \not=1$ and
$$
A'_m =    {\sigma_m \over 1 - \gamma_m}, \quad
B'_m =  (-1)^m { \prod_{h=0}^{m-1}\gamma_h^{a_{h+1}-b_{h+1}}\over (1-\gamma_m)\gamma_{m-1}}
$$
are well defined. Proposition \ref{xisxism} asserts that 
$$
\xi_{\bfs} (\beta, \alpha) = (1-\beta)\alpha  \bigl(A'_m + B'_m \xi_{\bfs_m} (\gamma_m, \gamma_{m-1}) \bigr). 
$$
It is sufficient to prove that 
$$
h(A'_m), h(B'_m) \ll_{\alpha, \beta} q_m
$$
to establish the proposition.
For $k = 0, \ldots , m-1$, we have
\begin{align*}
h \bigl( \prod_{h=0}^k\gamma_h^{a_{h+1}-b_{h+1}} \bigr) \le \sum_{h=0}^k a_{h+1} h(\gamma_h)
& \le \sum_{h=0}^k a_{h+1} (q_h h(\beta) + p_h h(\alpha)).  \\
& \le (q_{k+1} + q_{k}) h(\beta) + (p_{k+1} + p_{k}) h(\alpha), 
\end{align*}
by \eqref{sommeqk} and \eqref{sommepk}. This implies that $h(B'_m) \ll_{\alpha, \beta} q_m$. 
%Thus, if $\gamma_k$ and $\gamma_{k+1}$ are different from $1$, we get 
%\begin{equation}  \label{estimh} 
%h \biggl( {\prod_{h=0}^k\gamma_h^{a_{h+1}-b_{h+1}}
%\over (1-\gamma_{k+1})(1-\gamma_k)} \biggr) \ll_{\alpha, \beta} q_{k+1},
%\end{equation} 
%thus, taking $k = m-1$, we obtain $h(B'_m) \ll_{\alpha, \beta} q_m$. 
%Assume that there exists $\ell$ such that $\beta^{q_\ell} \alpha^{p_\ell} = 1$
%(otherwise, the corollary follows from the Proposition and the usual properties of the height) 

To estimate the height of
$$
\sigma_m = \sum_{n=1}^{q_{m}} s_n\beta^n\alpha^{\sum_{h=1}^n s_h},
$$
first note that its denominator is bounded from above by $q_m$ times the product of the denominators 
of $\alpha$ and $\beta$. Let $M \ge 2$ be an upper bound for the moduli of the conjugates of $\alpha$ and $\beta$. 
Then, the modulus of any conjugate of $\sigma_m$ is at most equal to $M^2 + \ldots + M^{2 q_m}$, thus 
less than $M^{2 q_m + 1}$. We conclude that $h(A'_m) \ll_{\alpha, \beta} q_m$, as asserted. 
\end{proof}

%\begin{theorem}
%Let $\theta$ and $\rho$ be real numbers with
%$0 \le \theta, \rho < 1$ and $\theta$ irrational.
%Let $\alpha$ and $\beta$ be complex numbers such that $0 < |\beta \alpha^\theta |< 1$. Then, 
%$\xi_{\s_{\theta,\rho}} (\beta, \alpha)$ is transcendental. 
%Then, 
%\end{theorem}

We are now equipped to complete the proof of Theorem \ref{maintrans}. 

\begin{proof}[Proof of Theorem \ref{maintrans} when $\theta$ has bounded partial quotients.] 
Let $\theta$ and $\rho$ be as in the statement of the theorem. 
Assume that $\theta$ has bounded partial quotients. 
Let $\alpha_1, \alpha_2$ be complex numbers such that $0 < |\alpha_1 \alpha_2^\theta |< 1$ and $\alpha_1 \not= 1$.

Let $M$ be a positive 
integer such that $a_k, b_k \le M$ for $k \ge 1$. 
Recall that $p_k / q_k$ denotes the $k$-th convergent to $\theta$. 
Let $m_0$ be as in Proposition \ref{chaine}. 
By compactness, there exist an increasing sequence $(\nu_k)_{k \ge 1}$ of positive integers, with $\nu_1 \ge m_0$, integers 
$g_1, g_2, \ldots , a'_1, a'_2, \ldots$ in $\{1, \ldots , M\}$, and integers $b'_1, b'_2, \ldots $ in $\{0, \ldots , M\}$ such that 
$$
(a_{\nu_k}, a_{\nu_k - 1}, a_{\nu_k-2}, \ldots ) \to (g_1, g_2, g_3, \ldots), \quad k \to \infty,     
$$
$$
(a_{\nu_k + 1}, a_{\nu_k + 2}, a_{\nu_k + 3}, \ldots  ) \to (a'_1, a'_2, a'_3, \ldots), \quad k \to \infty, 
$$
and 
$$
(b_{\nu_k + 1}, b_{\nu_k + 2}, b_{\nu_k + 3}, \ldots  ) \to (b'_1, b'_2, b'_3, \ldots), \quad k \to \infty. 
$$
As $k$ tends to infinity, the Sturmian word $\bfs_{\nu_k}$ with slope $\theta_{\nu_k}$ and 
intercept $b_{\nu_k + 1}, b_{\nu_k + 2}, \dots$ tends to the Sturmian word $\bfs'$ with slope $[0 ; a'_1, a'_2, \ldots ]$ and
formal intercept $b'_1, b'_2, \ldots$

Set 
$$
\phi := [0; g_1, g_2, \ldots ].
$$
%and let $\rho'$ be such that the sequence of digits in the $\theta'$-Ostrowski expansion of $\rho' - \theta'$ 
%is given by $(h_k)_{k \ge 1}$. 
Observe that $\phi$ is irrational and, by the theory of continued fractions, 
$$
\lim_{k \to + \infty} {p_{\nu_k - 1} \over p_{\nu_k}} = \lim_{k \to + \infty} {q_{\nu_k - 1} \over q_{\nu_k}} = \phi. 
$$
Define
$$
\Omega_k = 
\begin{pmatrix}
 q_{\nu_k} & p_{\nu_k} \\ q_{\nu_k - 1} & p_{\nu_k - 1} 
 \end{pmatrix}, \quad 
r_k = q_{\nu_k}, \quad k \ge 1.
$$
Assumption (i) is satisfied.

Since $0 < |\alpha_1 \alpha_2^\theta |< 1$, 
$$
\lim_{k \to + \infty} \, {\log |\alpha_1^{(k)} | \over r_k} = 
{q_{\nu_k} \log |\alpha_1| + p_{\nu_k} \log |\alpha_2| \over r_k} = \log |\alpha_1| + \theta \log |\alpha_2|, 
$$
$$
\lim_{k \to + \infty} \, {\log |\alpha_2^{(k)} | \over r_k} = 
{q_{\nu_k - 1} \log |\alpha_1| + p_{\nu_k - 1} \log |\alpha_2| \over r_k} = \phi (\log |\alpha_1| + \theta \log |\alpha_2| ). 
$$
and $\phi$ is irrational, Assumption (ii) is satisfied. 

Put
$$
f_0 (z_1, z_2) = \xi_{\bfs }(z_1, z_2), \quad f_k (z_1, z_2) = \xi_{\bfs_{\nu_k}} (z_1, z_2), \quad k \ge 1.
$$
%where $\bfs_{\nu_k}$ is the Sturmian word with slope $[0 ; a_{\nu_k + 1}, a_{\nu_k + 2}, \ldots ]$ and
%formal intercept $b_{\nu_k + 1}, b_{\nu_k + 2}, \ldots$ 
The coefficients of $f_k$ are in $\{0, 1\}$ for $k \ge 0$. 
It follows from Proposition \ref{chaine} that Assumption (iii) is satisfied. 

As noted above, we have
%Denoting by $\bfs'$ the Sturmian word with slope $[0 ; a'_1, a'_2, \ldots ]$ and
%formal intercept $b'_1, b'_2, \ldots$, we have
$$
\lim_{k \to \infty} f_k (z_1, z_2) = \xi_{\bfs' }(z_1, z_2).
$$
Furthermore, it follows from Proposition \ref{hecke} that $\xi_{\bfs' } (z_1, z_2)$ is a transcendental function over $\C (z_1, z_2)$. 

Let $p$ be a positive integer  
and $P_0 (z_1, z_2;  \us), \ldots , P_p (z_1, z_2; \us)$ be polynomials as in (iv). Let $E(z_1, z_2; \us)$ be as above.  
Setting
$$
\xi_{\bfs' } (z_1, z_2) = \sum_{\lambda_1, \lambda_2 \ge 0} \, \sigma_{\lambda_1, \lambda_2}  
z_1^{\lambda_1} z_2^{\lambda_2},
$$
we have
$$
\lim_{k \to + \infty} \, P_j(z_1, z_2; \usigma^{(k)}) = P_j (z_1, z_2; \usigma), \quad
\lim_{k \to + \infty} P_{\lambda_1, \lambda_2} (\usigma^{(k)}) = P_{\lambda_1, \lambda_2} (\usigma).
$$
If $P_j (z_1, z_2; \usigma)$, $0 \le j \le p$, are all zero, then the polynomials $P_j (z_1,z_2;\usigma^{(k)})$ vanish  identically   
for $k$ suficiently large, thus Assumption (iv) is clearly satisfied.   
%If $P_j (z_1, z_2; \usigma)$, $0 \le j \le p$, are not all zero, then 
Otherwise, 
$E(z_1, z_2 ; \usigma)$ is not zero, since $\xi_{\bfs' }$ is a transcendental function. 
Consequently, there exist nonnegative $\lambda_1, \lambda_2$ such that      
$P_{\lambda_1, \lambda_2} (\usigma)$ is nonzero. 
Hence, $P_{\lambda_1, \lambda_2} (\usigma^{(k)})$ is not zero for all $k$ 
sufficiently large and Assumption (iv) is satisfied. 
All this shows that Theorem \ref{LvdPchaines} applies and yields Theorem \ref{maintrans}   
when the slope $\theta$ has bounded partial quotients.  
\end{proof}

\end{document}